\documentclass[11pt,a4paper]{article}
\usepackage[a4paper,hmargin=2.8cm,vmargin=3cm]{geometry}
\usepackage[utf8x]{inputenc}
\usepackage{amsmath,amsthm,amsfonts,amssymb}
\usepackage{authblk}
\usepackage[bookmarksnumbered=true]{hyperref}
\usepackage{bbm}
\usepackage{bigints}
\usepackage[showonlyrefs]{mathtools}

\usepackage{color}
\usepackage{appendix}

\usepackage{graphicx}

\numberwithin{equation}{section}

\newtheorem{theorem}{Theorem}[section]
\newtheorem{corollary}[theorem]{Corollary}
\newtheorem{lemma}[theorem]{Lemma}
\newtheorem{proposition}[theorem]{Proposition}
\theoremstyle{definition}
\newtheorem{definition}[theorem]{Definition}
\newtheorem{remark}[theorem]{Remark}

\let\oldmarginpar\marginpar
\renewcommand\marginpar[1]{\-\oldmarginpar[\raggedleft\footnotesize #1]%
  {\raggedright\footnotesize #1}}

\def\cD{\mathcal{D}}
\def\cM{\mathcal{M}}
\def\cB{\mathcal{B}}

\def\N{\mathbb{N}}

\def\R{\mathbb{R}}
\let\e=\varepsilon

\let\.=\cdot
\let\0=\emptyset

\def\square{\hbox{$\sqcap\kern-7pt\sqcup$}}

\def\dev{\operatorname{div}}

\def\be{\begin{equation}}
\def\ee{\end{equation}}
\def\bea{\begin{eqnarray}}
\def\eea{\end{eqnarray}}

\def\ˆ{^{•}}

\newcommand{\p}{\partial} 
 
\newcommand{\la}{\lambda} 
 
\newcommand{\IR}{\mathbb{R}}

\newcommand{\IL}{\mathcal{L}}

\title{Lagrangian solutions to the Vlasov-Poisson system \\ with a point charge}

\author{Gianluca Crippa, Silvia Ligabue, Chiara Saffirio}

\def\adresse{
\begin{description}
\item[G. Crippa:] 
Departement Mathematik und Informatik,\\ Universit\"{a}t Basel,
Spiegelgasse 1,
CH-4051 Basel, Switzerland\\
E-mail: \texttt{gianluca.crippa@unibas.ch }

\item[S. Ligabue:] Departement Mathematik und Informatik,\\ Universit\"{a}t Basel,
Spiegelgasse 1,
CH-4051 Basel, Switzerland\\
E-mail: \texttt{silvia.ligabue@unibas.ch}

\item[C. Saffirio:] Institute of Mathematics, \\ University of Z\"urich, Winterthurerstrasse 190, CH-8057 Z\"urich, Switzerland\\
E-mail: \texttt{chiara.saffirio@math.uzh.ch}

\end{description}
}

\date{\today}

\begin{document}

\maketitle

\begin{abstract}
We consider the Cauchy problem for the repulsive Vlasov-Poisson system in the three dimensional space, where the initial datum is the sum of a diffuse density, assumed to be bounded and integrable, and a point charge. Under some decay assumptions for the diffuse density close to the point charge, under bounds on the total energy, and assuming that the initial total diffuse charge is strictly less than one, we prove existence of global Lagrangian solutions. Our result extends the Eulerian theory of~\cite{DMS}, proving that solutions are transported by the flow trajectories. The proof is based on the ODE theory developed in~\cite{BBC1} in the setting of vector fields with anisotropic regularity, where some components of the gradient of the vector field is a singular integral of a measure. 
\end{abstract}


\section{Introduction and main results}

We study the Cauchy problem associated with the Vlasov-Poisson system
\begin{equation}\label{eq:VP} \left\{
\begin{array}{l}
\partial_t f + v\cdot\nabla_x f + E\cdot\nabla_v f = 0\,,\\\\
E(t,x)=\nabla(\frac{\gamma}{|  \cdot  |}*\rho)(t,x)\,,\\\\
\rho(t,x)=\int f(t,x,v)\,dv\,,
\end{array}
\right.
\end{equation}
in the three dimensional space, where $f:\R_+\times\R^3\times\R^3\to\R_+$ stands for the non-negative density of particles in a plasma under the effect of a self-induced field $E$, while $\rho:\R_+\times\R^3\to\R_+$ is the spatial density and $\gamma\in\{-1,1\}$ is a parameter which models the repulsive ($\gamma=1$) or attractive ($\gamma=-1$) nature of the particles. We recall that the self-induced field $E(t,x)$ is a conservative force. Therefore there exists a function $U:\R_+\times\R^3\to\R$ such that $E(t,x)=\nabla_x U(t,x)$, thus the Poisson equation $-\Delta U=\rho$ is fulfilled.  In other words, we can rewrite the system \eqref{eq:VP} as a Vlasov equation coupled with a Poisson equation, from which the name Vlasov-Poisson arises.  
From a physical viewpoint, the repulsive case represents the evolution of charged particles in presence of their self-consistent electric field and it is used in plasma physics or in semi-conductor devices. The attractive case describes the motion of galaxy clusters under the gravitational field with many applications in astrophysics. In this paper we focus on the repulsive case, by fixing $\gamma=1$  in \eqref{eq:VP}.

In the last decades the Vlasov-Poisson system \eqref{eq:VP} has been largely investigated.  
Existence of classical solutions under regularity assumptions on the initial data goes back to Iordanski \cite{Ior} in dimension one and to Okabe and Ukai \cite{UO} in dimension two. The three dimensional case has been addressed first by Bardos and Degond \cite{BD} for small initial data, and then extended to a more general class of initial plasma densities by Pfaffelm{o}ser \cite{Pf} and by Lions and Perthame \cite{LP}. Improvements in three dimensions have been obtained in \cite{Sch, W, castella, L, CZ}. Global existence of weak solutions has been studied by Arsenev \cite{A} for bounded initial data with finite kinetic energy, while the global existence of renormalized solutions is due to Di Perna and Lions \cite{DiLi}, assuming finite total energy and $f_0\in  L\log L(\R^3\times\R^3)$. The latter assumption has been recently relaxed to $f_0\in L^1(\R^3\times\R^3)$ in \cite{ACF} and \cite{BBC2}. 

One might wonder what happens when $f_0\notin L^1(\R^3\times\R^3)$.  
In this paper we shall address this question by assuming $f_0$ to be the sum of an integrable bounded plasma density and a Dirac mass. This is equivalent to studying the Cauchy problem associated with the following system:
\begin{equation}\label{eq:VP-dirac}\left\{
\begin{array}{l}
\partial_t f + v\cdot\nabla_x f + (E+F)\cdot\nabla_v f = 0\,,\\\\
E(t,x)=\int\frac{x-y}{|x-y|^3}\rho(t,y)\,dy\,,\\\\
\rho(t,x)=\int f(t,x,v)\,dv\,,\\\\
F(t,x)=\frac{x-\xi(t)}{|x-\xi(t)|^3}\,,
\end{array}
\right.
\end{equation}
where the singular electric field $F:=F(t,x)$ is induced by a point charge located at a point $\xi(t)$, whose evolution is given by the Newton equations:
 \begin{equation}\label{eq:newton}
 \left\{
 \begin{array}{l}
 \dot{\xi}(t)=\eta(t)\,,\\\\
 \dot{\eta}(t)=E(t,\xi(t))\,. 
 \end{array}
 \right.
 \end{equation}

For every $(x,v)\in\R^3\times\R^3$, we denote by $f_0(x,v)=f(0,x,v)$ and by $(\xi_0,\eta_0)=(\xi(0),\eta(0))$  respectively  the initial density and initial state of the point charge in the phase space $\R^3\times\R^3$. 
The system \eqref{eq:VP-dirac}-\eqref{eq:newton} can be formally rewritten in the form \eqref{eq:VP} for the total density $f(t)+\delta_{\xi(t)}\otimes\delta_{\eta(t)}$.

The model \eqref{eq:VP-dirac}--\eqref{eq:newton} has been recently introduced by Caprino and Marchioro in \cite{CM}, where they have shown global existence and uniqueness of classical solutions in two dimensions. This result has been extended to the three dimensional case in \cite{MMP} by Marchioro, Miot and Pulvirenti. Both \cite{CM} and \cite{MMP} require that the initial plasma density does not overlap the point charge. This assumption has been relaxed in \cite{DMS}, where weak solutions of the system \eqref{eq:VP-dirac}--\eqref{eq:newton} have been obtained for initial data which may overlap the point charge, but do have to decay close to it. The price to pay is that the solution is no longer known to be unique nor Lagrangian.
In the following we will call {\em Lagrangian solution} a plasma density $f$ and a trajectory $(\xi,\eta)$ of the Dirac mass, both defined for $t \in \R_+$, such that $f$ is transported by the Lagrangian flow $(X,V)$, solution to the ODE-system  
\begin{equation}\label{eq:charXV}
\begin{cases}
\dot{X}(t,x,v) = V(t,x,v) \\
\dot{V}(t,x,v) = E ( t, X(t,x,v)) + F(t,X(t,x,v)) \\
\big( X(0,x,v) , V(0,x,v) \big) = (x,v) \,,
\end{cases}
\end{equation}
more precisely
$$
f(t,x,v) = f_0 \big( X^{-1}(t,\cdot,\cdot)(x,v) \,,\, V^{-1}(t,\cdot,\cdot)(x,v) \big) \,.
$$
This is a finer physical structural information on the solution than the mere fact that $f$ and $(\xi,\eta)$ are weak solutions of~\eqref{eq:VP-dirac}--\eqref{eq:newton}.

In the framework of classical solutions, the Eulerian description and the Lagrangian evolution of particles given by the system of characteristics are completely equivalent.
When dealing with weak or renormalized solutions, the correspondence between the Eulerian and Lagrangian formulations is non trivial and requires a careful analysis of the Lagrangian structure of transport equations with non-smooth vector fields. Indeed, without any regularity assumptions, it is not even  clear whether the flow associated with the vector field generated by a weak solution exists. 
 
In recent years the theory of transport and continuity equations with non-smooth vector fields has witnessed a massive amount of progress, also due to the large number of applications to nonlinear PDEs. In the seminal paper by DiPerna and Lions~\cite{DiLi} the theory has been first developed in the context of Sobolev vector fields, with suitable bounds on space divergence and under suitable growth assumptions. This has been extended by Ambrosio~\cite{Amb} to the setting of vector field with bounded variation ($BV$), roughly speaking allowing for discontinuities along codimension-one hypersurfaces. See also~\cite{HW} for an up-to-date survey of this theory and its recent advances.

In the context of the Vlasov-Poisson system with a Dirac mass considered in this paper (\eqref{eq:VP-dirac}-\eqref{eq:newton}) the system of characteristics is given by \eqref{eq:charXV}.
The singular electric field $F$ generated by the Dirac mass is not regular, and it does not even belong to any Sobolev space of order one or to the $BV$ space. Therefore the theory of~\cite{DiLi,Amb} cannot be directly applied to this case. However, a related theory of Lagrangian flows for non-smooth vector fields has been initiated in~\cite{CDL}. In a nutshell, the approach in~\cite{CDL} provides a suitable extension of Gr\"{o}nwall-like estimates to the context of Sobolev vector fields, by introducing a suitable functional measuring a logarithmic distance between Lagrangian flows. In addition, the theory in~\cite{CDL} has a quantitative character, providing explicit rates in the stability and compactness estimates, and it has been pushed even to situations out of the Sobolev or $BV$ contexts of~\cite{DiLi,Amb}. In particular, using more sophisticate harmonic analysis tools, the case when the derivative of the vector field is a singular integral of an $L^1$ function has been considered in~\cite{BC}. This has been further developed in~\cite{BBC1}, allowing for singular integrals of a measure, under a suitable condition on splitting of the space in two groups of variables, modelled on the situation for the Vlasov-Poisson characteristics~\eqref{eq:charXV}. This theory has been applied to the study of the Euler equation with $L^1$ vorticity~\cite{BBC3} and of the Vlasov-Poisson equation with $L^1$ density~\cite{BBC2}. The latter has also been studied in~\cite{ACF}, using the theory of maximal Lagrangian flows developed in~\cite{ACF2}.

The purpose of this paper is to recover the relation between the Eulerian and the Lagrangian picture for solutions provided in \cite{DMS} by exploiting the transport structure of the equation. In other words we aim to prove existence of Lagrangian solutions to the Vlasov-Poisson system \eqref{eq:VP} with $\gamma = 1$ and initial data $f_0+\delta_{\xi_0}\otimes\delta_{\eta_0}$, where $f_0$ satisfies the assumptions of \cite{DMS}. 

Our main result is the following

\begin{theorem}
\label{thm:main}
Let $f_0\in L^1\cap L^\infty(\R^3\times\R^3)$, such that the initial total charge
\begin{equation}\label{eq:M(0)}
M(0)=\iint f_0(x,v)\,dxdv<1
\end{equation}
 and the total energy 
 \begin{equation}\label{eq:H(0)}
 H(0)=\iint \frac{|v|^2}{2}f_0(x,v)dxdv+ \frac{|\eta_0|^2}{2}+ \frac{1}{2}\iint \frac{\rho(0,x)\rho(0,y)}{|x-y|}dxdy+\iint \frac{\rho(0,x)}{|x-\xi_0|}dx
\end{equation}  is finite.  
Assume that there exists $m_0>6$ such that for all $m<m_0$ the energy moments
\begin{equation}\label{eq:moments}
\mathcal{H}_m(0)=\iint\left(|v|^2+\frac{1}{|x-\xi_0|}\right)^{m/2}\,f_0(x,v)dxdv
\end{equation}
are finite.
Then there exists a global Lagrangian solution to the system \eqref{eq:VP-dirac}--\eqref{eq:newton}. 
\end{theorem}

\noindent Some remarks are in order: 
\begin{enumerate}
\item The moments \eqref{eq:moments} are propagated in time (see Proposition \ref{prop:moments} for the precise statement and \cite{DMS} for details). This implies $f\in{C}(\R_+,L^p(\R^3\times\R^3))\cap L^{\infty}(\R_+,L^{\infty}(\R^3\times\R^3))$ for $1\leq p <\infty$, $E\in L^{\infty}([0,T],{C}^{0,\alpha}(\R^3))$ for some $\alpha\in(0,1)$ (see Remark \ref{remark:rho}) and $\xi\in{C}^2(\R_+)$. 
\item We observe that the hypothesis \eqref{eq:M(0)} is needed only to get a control on the electric field generated by the point charge (see Proposition \ref{prop:viriale}).  From the viewpoint of physics, this is a purely technical and too restrictive condition. In a forthcoming paper, we plan to remove this constraint.   
\item When considering the Cauchy problem associated with \eqref{eq:VP} with $\gamma=-1$ (attractive case) and initial data $f_0+\delta_{\xi_0}\otimes\delta_{\eta_0}$,  the whole strategy fails. This is due to a crucial change of sign in the total energy $H$ and in $\mathcal{H}_m$. More precisely, the last two terms in \eqref{eq:H(0)} and the last term in \eqref{eq:moments}, representing respectively the potential energy of the system and the potential energy per particle, come with a negative sign.   
This prevents to establish a control on the trajectory of the point charge as in Proposition \ref{prop:charge-bounds} and to prove Proposition \ref{prop:moments}. \\
The simpler case of a system in which the particles in the plasma are interacting through a repulsive potential while the point charge generates an attractive force field has been treated in \cite{CMMP} in dimension two. Notice that, even in this case, the existence of solutions in three dimensions remains an interesting open problem.  
\item Theorem \ref{thm:main} does not imply uniqueness of the Lagrangian solution. In analogy to \cite{Pf}, where uniqueness of compactly supported classical solutions of \eqref{eq:VP} has been proved, uniqueness of solutions to \eqref{eq:VP-dirac}-\eqref{eq:newton} which do not overlap with the point charge and have compact support in phase space has been established  in \cite{MMP}. In the context of weak solutions to \eqref{eq:VP}, sufficient conditions for uniqueness have been proved in \cite{LP} and later extended to weak measure-valued solutions  with bounded spatial density  by Loeper \cite{L}. Recently Miot \cite{M} generalised the latter condition to a class of solutions whose $L^p$ norms of spatial density grow at most linearly w.r.t. $p$, then extended to spatial densities belonging to some Orlicz space in \cite{HM}. Unfortunately, it seems that none of these conditions apply to our setting and new ideas are needed.      
\end{enumerate}
 
Let us informally describe the main steps of our proof. We rely on the result in~\cite{MMP}, which guarantees existence of a (unique) Lagrangian solution to the Cauchy problem for the Vlasov-Poisson system \eqref{eq:VP-dirac}-\eqref{eq:newton}, provided that at initial time the plasma density has a positive distance from the Dirac mass and bounded support in the phase space. We therefore approximate the plasma density $f_0$ at initial time by a sequence $f_0^n$ obtained by cutting off $f_0$ close to the Dirac mass in the space variable and out of a compact set in phase space. We use~\cite{MMP} to construct a Lagrangian flow $(X_n,V_n)$ and a trajectory for the Dirac mass $(\xi_n,\eta_n)$ corresponding to the initial data $f_0^n$ and $(\xi_0,\eta_0)$. The assumptions of Theorem~\ref{thm:main} together with the propagation of the moments $\mathcal{H}_m$ from~\cite{DMS} entail some additional integrability of the densities $\rho_n$, which in turn implies uniform H\"older estimates on the electric fields $E_n$. Moreover, assumption~\eqref{eq:M(0)} allows to prove some uniform decay of the superlevels of the Lagrangian flows $(X_n,V_n)$, which combined with an extension of the Lagrangian theory developed in~\cite{BBC1} gives compactness of the Lagrangian flows $(X_n,V_n)$. Finally, standard energy estimates guarantee the uniform continuity of the trajectories $\xi_n$ uniformly in $n$. All this enables us to pass to the limit in the Lagrangian formulation of the problem, eventually giving a Lagrangian solution corresponding to the initial plasma density $f_0$. 

One of the main technical difficulties of our analysis is the control on large velocities. In this work, this reflects in the necessity of some control on the superlevels of the Lagrangian flows (see Definition \ref{defGlambda}). This was already an issue in \cite{BBC2} and here the situation is made even more complicated by the presence of the singular field generated by the point charge. We tackle this problem by weighting superlevels with the measure given by the initial distribution of charges $f_0(x,v)\,dx\,dv$ (see Lemma \ref{estsuper}). In this way the control on the superlevels can be proven exploiting virial type estimates on the time integral of the electric field generated by the diffuse charge and evaluated in the point charge (see Proposition \ref{prop:viriale}). This carries the physical meaning that it is only relevant to control the flow starting from points in the support of the initial density of charge.

In connection to the theory of \cite{BBC1}, this weighted estimates manifest in the presence of the density $h=f_0$ in the functional \eqref{eq:functional} measuring the compactness of the flows. Moreover, in contrast to \cite{BBC2}, which was based on the isotropic analysis of \cite{BC}, here we strongly rely on the anisotropic theory of \cite{BBC1} in which some components of the gradient of the velocity field are allowed to be singular integrals of measures, accounting for the presence of the point charge.  

Notice that in the somewhat related case of the vortex wave system  a similar analysis has been carried out in \cite{CLFMNL}. 
In that context the vector field does not enjoy an anisotropic structure, but the singularity can be dealt with exploiting the specific form of the singular part of the electric field.

The plan of the paper is the following:  in Section \ref{sec:3} we present and prove the key theorem on  Lagrangian flows; in Section \ref{sec:2} we recall some useful properties related to solutions of the Vlasov-Poisson system; in Section \ref{sec:4} we give the proof of Theorem \ref{thm:main}, which follows from compactness arguments by using the results established in Section \ref{sec:3} and \ref{sec:2}.

\medskip

\noindent {\bf Acknowledgements.} The authors are grateful to the anonymous referee for suggesting a simplification in the proof of Lemma \ref{estsuper} which led to a substantial shortening of the argument.
GC and SL are partially supported by the Swiss National Science Foundation
grant 200020\_156112 and by the ERC Starting Grant 676675 FLIRT. 
CS is supported by the Swiss National Science Foundation through the Ambizione grant S-71119-02.


\section{Lagrangian flows}\label{sec:3}

Consider a smooth solution $u$ to a transport equation in $\R^+\times\R^d$
\begin{equation*}
\partial_t u+ b\cdot\nabla_z u=0\,,
\end{equation*}
where $b=b(t,z)$ is a smooth vector field.
Then $u$ is constant along the characteristics $s\mapsto Z(s,t,z)$, exiting from $z$ at time $t$, i.e. solutions to the equation

\begin{equation}\label{eq:characteristicsystem}
\dfrac{dZ}{ds}(s,t,z)=b(s,Z(s,t,z)),
\end{equation}
with initial data $Z(t,t,z)=z$. Thus the solution can be expressed as $u(t,z)=u_0(Z(0,t,z))$. 

For simplicity from now on we will consider the initial time $t$ in \eqref{eq:characteristicsystem} fixed and denote the flow $Z(s,t,z)$ by $Z(s,z)$.

In this paper we deal with flows of non-smooth vector fields. In order to extend the usual notion of characteristics to our case, we extend the definition of \textit{regular Lagrangian flows} in a renormalized sense by introducing a reference measure with bounded density. This turns out to be convenient in the estimates involving the superlevels of the flow (see Lemma \ref{estsuper}).

\begin{definition}[$\mu$-regular Lagrangian flow]\label{def:RLF}

Given an absolutely continuous measure $\mu$ with bounded density, a vector field $b(s,z): [0,T] \times \R^{d}\rightarrow \R^{d}$, and $t\in [0,T)$, a map 
\[ Z=Z(s,z) \in C([t,T]_{s}; L^{0}_{\rm loc}(\R^d_z,d\mu)) \cap \cB([t,T]_{s}; \log\log L_{\rm loc}(\R^d_z,d\mu)) \]
is a $\mu$-regular Lagrangian flow in the renormalized sense starting at time $t$ relative to $b$ if we have the following:
\item[(1)] The equation
\be\label{eq:RSder} \p_s(\beta(Z(s,z)))=\beta^{'}(Z(s,z))b(s,Z(s,z))    \ee
holds  in $\cD ^{'}((t,T))$ for $\mu$-a.e. $z$, for every function $\beta \in C^{1}(\R^{d};\R)$ that satisfies 
\begin{equation*}
|\beta(z)| \leq C(1+\log(1+\log(1+|z|^2)))\quad \text{and}\quad |\beta^{'}(z)| \leq \dfrac{C\,|z|}{(1+|z|^2)(1+\log(1+|z|^2))}
\end{equation*} 
for all $z \in \R^{d}$;
\item[(2)] $Z(t,z)=z$ for $\mu$-a.e. $z\in \R^{d}$;
\item[(3)] There exists a $L \geq 0$, called compressibility constant, such that, for every $s\in [t,T]$,
\begin{equation}
Z(s,\cdot)_{\#}\mu \leq L \mu,
\end{equation}
i.e.
\[\mu(\{z\in \R^{d} :\, Z(s,z)\in B\})\leq L\mu(B) \qquad \mbox{for every Borel set } B \subset \R^{d}.  \]


\end{definition}

We have denoted with $L^{0}_{\rm loc}$ the space of measurable functions endowed with the local convergence in measure, by $\log\log L_{\rm loc}$ the space of measurable functions $u$ such that\break $ \log(1+\log(1+|u|^2))$ is locally integrable, and by $\cB$ the space of bounded functions. When the reference measure $\mu$ is not explicitly specified, the spaces under consideration are endowed with the Lebesgue measure. 

\begin{remark}

\noindent Our definition of $\mu$-regular Lagrangian flow slightly differs from the one in \cite{BBC1}. On the one hand we change the reference measure from the Lebesgue measure to $\mu$. On the other hand we consider a different class of $\beta$'s, which grow slower at infinity.
\end{remark}

\begin{definition}\label{defGlambda}
Let $Z:[t,T]\times\R^d\to\R^d$ be a measurable map. For every $\lambda>0$, we define the sublevel of  $Z$ as
\begin{equation}
 G_{\lambda}=\left\{ z\in \R^{d}: |Z(s,z)|\leq \lambda \text{ for almost all } s \in [t,T] \right\}.  
 \end{equation}
\end{definition}

\subsection{Setting and result of \cite{BBC1}}

We summarize here the regularity setting and the stability estimate of \cite{BBC1}. We say that a vector field $b$ satisfies \textbf{(R1)} if $b$ can be decomposed as
\begin{equation}
\frac{b(t,z)}{1+|z|}=\tilde{b}_1(t,z)+ \tilde{b}_2(t,z)
\end{equation}
where $\tilde{b}_1 \in L^{1}((0,T);L^{1}(\R^{d}))$, $\tilde{b}_2 \in L^{1}((0,T);L^{\infty}(\R^{d}))$.
Notice that this hypothesis leads to an estimate for the decay of the superlevels of a regular Lagrangian flow. In fact Lemma~3.2 of \cite{BBC1} tells us that, if $b$ satisfies (R1) and $Z$ is a regular Lagrangian flow associated with $b$ starting at time $t$, with compressibility constant $L$, then $\IL^{d}(B_r\setminus G_{\lambda}) \leq g(r,\lambda)$ for any $r,\lambda >0$, where $g$ depends only on $L$, $\|\tilde{b}_1\|_{L^{1}((0,T);L^{1}(\R^{d}))}$ and $\|\tilde{b}_2\|_{L^{1}((0,T);L^{\infty}(\R^{d}))}$ and satisfies $g(r,\lambda)\downarrow 0$ for $r$ fixed and $\lambda \uparrow \infty$.

\textbf{(R2)} We want to consider a vector field $b(t,z)$ such that its regularity changes with respect to different directions of the variable $z \in \R^{d}$, that is we consider $\R^{d}=\R^{n_1}\times \R^{n_2}$ and $z=(z_1,z_2)$ with $z_1\in \R^{n_1}$ and $z_2\in\R^{n_2}$. We denote with $D_1$ the derivative with respect to $z_1$ and $D_2$ the derivative with respect to $z_2$. Accordingly we denote $b=(b_1,b_2)(s,z)\in \R^{n_1}\times \R^{n_2}$ and $Z=(Z_1,Z_2)(s,z)\in \R^{n_1}\times \R^{n_2}$. Therefore we assume that the elements of the matrix $Db$, denoted as $(Db)^{i}_j$, are in the form
\begin{equation}
(Db)^{i}_j = \sum_{k=1}^{m} \gamma_{jk}^{i}(s,z_2)S_{jk}^{i}\textsf{m}_{jk}^{i}(s,z_1) 
\end{equation}
where 
\begin{itemize}
\item[-] $S_{jk}^{i}$ are singular integral operators associated with singular kernels of fundamental type in $\R^{n_1}$ (see \cite{Stein}),
\item[-] the functions $\gamma_{jk}^{i}$ belong to $L^{\infty}((0,T);L^{q}(\R^{n_2}))$ for some $q >1$,
\item[-] $\textsf{m}_{jk}^{i} \in L^{1}((0,T);L^{1}(\R^{n_1}))$ for all the elements of the submatrices $D_1b_1$, $D_2b_1$ and $D_2b_2$, while $\textsf{m}_{jk}^{i} \in L^{1}((0,T);\cM(\R^{n_1}))$ if $(Db)^{i}_j$ is an element of  $D_1b_2$.
\end{itemize}
We have denoted by $L^{1}((0,T);\cM(\R^{n_1}))$ the space of all functions $t\mapsto\mu(t,\cdot)$ taking values in the space $\cM(\R^{n_1})$ of finite signed measures on $\R^{n_1}$ such that
\[ \int_0^{T} \|\mu(t,\cdot)\|_{\cM(\R^{n_1})}dt < \infty.  \]

Moreover, we assume condition \textbf{(R3)}, that is
\begin{equation}
b \in L^{p}_{\rm loc}([0,T]\times \R^{d}) \,\,\,\,\,\,\,\,\, \text{for some }p>1.
\end{equation}

We recall the main theorem from \cite{BBC1}.

\begin{theorem} \label{teo1}
Let $b$ and $\bar{b}$ be two vector fields satisfying assumption (R1), where $b$ satisfies also (R2), (R3). Fix $t\in [0,T]$ and let $Z$ and $\bar{Z}$ be regular Lagrangian flows starting at time $t$ associated with $b$ and $\bar{b}$ respectively, with compressibility constants $L$ and $\bar{L}$. Then the following holds. For every $\gamma$, $r$, $\eta>0$ there exist $\lambda$, $C_{\gamma,r,\eta}>0$ such that
\[ \IL^{d}\left(B_r \cap \{|Z(s,\cdot)-\bar{Z}(s,\cdot)|>\gamma\}\right) \leq C_{\gamma,r,\eta}\|b-\bar{b}\|_{L^{1}((0,T)\times B_{\lambda})} + \eta \]
for all $s\in [t,T]$. The constants $\lambda$ and $C_{\gamma,r,\eta}$ also depend on:  
\begin{itemize}
\item The equi-integrability in $L^{1}((0,T);L^{1}(\R^{n_1}))$ of all the $\textsf{m}_{jk}^{i}$ which belong to this set, as well as the norm in $L^{1}((0,T);\cM(\R^{n_1}))$ of the remaining $\textsf{m}_{jk}^{i}$ (where these functions are associated with $b$ as in (R2)),
\item The norms of the singular integrals operators $S_{jk}^{i}$, as well as the norms of $\gamma_{jk}^{i}$ in $L^{\infty}((0,T);L^{q}(\R^{n_2}))$  (associated with $b$ as in (R2)),
\item The norm in $L^{p}((0,T)\times B_{\lambda})$ of $b$,
\item The $L^{1}((0,T);L^{1}(\R^{d})) + L^{1}((0,T);L^{\infty}(\R^{d}))$ norms of the decomposition of $b$ and $\bar{b}$ as in (R1), 
\item The compressibility constants $L$ and $\bar{L}$.
\end{itemize}
\end{theorem}

\subsection{Flow estimate in the new setting}
We are going now to state a variant of this theorem, where (R1) and (R2) are replaced by (R1a) and (R2a) below. The dimension $d$ will be here equal to $2N$, instead of $n_1+n_2$, and the variable $z$ will be in the form $z=(x,v)\in\R^N\times\R^N$.

We consider the following assumptions, that are adapted to our setting of the Vlasov-Poisson system with a point charge:

\textbf{(R1a)} For all $\mu$-regular Lagrangian flow $Z: [t,T]\times \R^{2N}\rightarrow \R^{2N}$ relative to $b$ starting at time $t$ with compression constant $L$, and for all $r$, $\lambda >0$,
\begin{equation}
\mu(B_r\setminus G_{\lambda}) \leq g(r,\lambda), \,\,\,\,\,\,\,\,\text{with }g(r,\lambda)\rightarrow 0\text{ as }\lambda \rightarrow \infty\text{ at fixed } r,
\end{equation} 
where $G_{\lambda}$ denotes the sublevel of the flow $Z$ defined in \eqref{defGlambda}. 

\textbf{(R2a)} Motivated by the particular structure of the Vlasov-Poisson system \eqref{eq:VP-dirac}-\eqref{eq:newton}, we assume $b$ to have the following structure:  
\begin{equation}
b(t,x,v)=(b_1,b_2)(t,x,v)=(b_1(v),b_2(t,x)),
\end{equation}
with
\begin{equation}
b_1 \in \text{Lip}(\R^{N}_v),
\end{equation} 
and where $b_2$ is such that for every $j=1,\ldots,N$,
\begin{equation}
\p_{x_j}b_2 = \sum_{k=1}^{m}S_{jk} \textsf{m}_{jk},
\end{equation}
where $S_{jk}$ are singular integrals of fundamental type on $\R^{N}$ and $\textsf{m}_{jk} \in L^{1}((0,T);\cM(\R^{N}))$.

Notice that assumption \textbf{(R2a)} does not imply assumption  \textbf{(R2)}, in which it is relevant that all components of $Db$ are singular integrals of {\it finite} measures. 


\begin{theorem}\label{teo2}
Let $\mu=h\, \mathcal{L}^{2N}$ with $h\in L^{1}\cap L^{\infty}$ and non-negative. Let $b$ and $\bar{b}$ be two vector fields satisfying (R1a), $b$ satisfying also (R2a), (R3). Given $t\in [0,T]$, let $Z$ and $\bar{Z}$ be $\mu$-regular Lagrangian flows starting at time $t$ associated with $b$ and $\bar{b}$ respectively, with sublevels $G_{\lambda}$ and $\bar{G}_{\la}$, and compressibility constants $L$ and $\bar{L}$. Then the following holds.\newline
For every $\gamma,\,r,\,\eta >0$, there exist $\lambda$, $C_{\gamma, r, \eta}>0$ such that
\begin{equation*}
\mu(B_r \cap \{|Z(s,\cdot )- \bar{Z}(s,\cdot )|>\gamma\} )\leq C_{\gamma, r, \eta} \|b-\bar{b}\|_{L^{1}((0,T)\times B_{\lambda})}+ \eta
\end{equation*}
uniformly in $s,\,t\in [0,T]$. The constants $\lambda$ and $C_{\gamma, r, \eta}$ also depend on:
\begin{itemize}
\item The norms of the singular integral operators $S_{jk}$ from {(R2a)},
\item The norms in $L^{1}((0,T);\cM(\R^{N}))$ of $\textsf{m}_{jk}$  from {(R2a)},
\item The Lipschitz constant of $b_1$ from {(R2a)},
\item The norm in $L^{p}((0,T)\times B_{\lambda})$ of $b$ corresponding to {(R3)},
\item The rate of decay of $\mu(B_r\setminus G_{\la})$ and $\mu(B_r \setminus \bar{G}_{\la})$ from {(R1a)},
\item The norm in $L^{\infty}(\R^{2N})$ of the function $h$ defined in {(R1a)},
\item The compressibility constants $L$ and $\bar{L}$.
\end{itemize}
\end{theorem}

\begin{proof}
The proof follows the same line as in Theorem \ref{teo1} (see \cite{BBC1}), with some modifications due to the different hypotheses. Given $\delta_1$, $\delta_2>0$, let $A$ be the constant $2N\times 2N$ matrix
\[ A= \text{Diag}(\underbrace{\delta_1,\ldots,\delta_1}_{\text{$N$ times}},\underbrace{\delta_2,\ldots,\delta_2}_{\text{$N$ times}}), \]
that means $A(x,v)=(\delta_1 x, \delta_2 v)$. We consider the following functional depending on the two parameters $\delta_1$ and $\delta_2$, with $\delta_1\leq \delta_2$:
\begin{equation}\label{eq:functional}
\Phi_{\delta_1,\delta_2}(s)= \iint\limits_{B_r\cap G_{\la}\cap \bar{G}_{\la}} \text{log}(1+|A^{-1}[Z(s,x,v)-\bar{Z}(s,x,v)]|)\, h(x,v)\, dx\,dv.
\end{equation}
In order to improve the readability of the following estimates, we will use the notation ``$\lesssim $" to denote an estimate up to a constant only depending on absolute constants and on the bounds assumed in Theorem \ref{teo2}, and the notation ``$\lesssim_{\la}$" to mean that the constant could also depend on the truncation parameter for the superlevels of the flow $\la$. The norm of the measure $\textsf{m}$ however will be written explicitly.
\medskip

\noindent \textit{Step 1: Differentiating $\Phi_{\delta_1,\delta_2}.$} Differentiating with respect to time and taking out of the integral the $L^{\infty}$ norm of $h$, we get
\begin{align*}
\Phi_{\delta_1,\delta_2}'(s) &\leq \|h\|_{L^{\infty}(\R^{2N})}  \iint\limits_{B_r\cap G_{\la}\cap \bar{G}_{\la}}  \frac{|A^{-1}[b(s,Z(s,x,v))-\bar{b}(s,\bar{Z}(s,x,v))]|}{1+|A^{-1}[Z(s,x,v)-\bar{Z}(s,x,v)]|}\, dx\,dv \\
& \lesssim  \iint\limits_{B_r\cap G_{\la}\cap \bar{G}_{\la}}  \frac{|A^{-1}[b(s,Z(s,x,v))-\bar{b}(s,\bar{Z}(s,x,v))]|}{1+|A^{-1}[Z(s,x,v)-\bar{Z}(s,x,v)]|}\, dx\,dv\,.
\end{align*}
Then we set $Z(s,x,v)=Z$ and $\bar{Z}(s,x,v)=\bar{Z}$ and we estimate
\[ \Phi_{\delta_1,\delta_2}'(s) \lesssim  \iint\limits_{B_r\cap G_{\la}\cap \bar{G}_{\la}} |A^{-1}[b(s,\bar{Z})-\bar{b}(s,\bar{Z})]|\,dx\,dv +  \iint\limits_{B_r\cap G_{\la}\cap \bar{G}_{\la}}  \frac{|A^{-1}[b(s,Z)-b(s,\bar{Z})]|}{1+|A^{-1}[Z-\bar{Z}]|}\, dx\,dv.   \]
After a change of variable along the flow $\bar{Z}$ in the first integral, and noting that $\delta_1 \leq \delta_2$, we further obtain
\begin{align*}
\Phi_{\delta_1,\delta_2}'(s) \lesssim \frac{\bar{L}}{\delta_1}&\|b(s,\cdot)-\bar{b}(s,\cdot)\|_{L^{1}(B_{\la})} \\
&+  \iint\limits_{B_r\cap G_{\la}\cap \bar{G}_{\la}}  \text{min} \left\lbrace |A^{-1}[b(s,Z)-b(s,\bar{Z})]|,\frac{|A^{-1}[b(s,Z)-b(s,\bar{Z})]|}{|A^{-1}[Z-\bar{Z}]|}  \right\rbrace  dx\,dv\,.
\end{align*}

\noindent \textit{Step 2: Splitting the quotient.} Using the special form of $b$ from (R2a) and the action of the matrix $A^{-1}$, we have
\[ A^{-1}[Z-\bar{Z}] = \left(\frac{X-\bar{X}}{\delta_1}, \frac{V-\bar{V}}{\delta_2}\right)\]
and
\[ A^{-1}[b(s,Z)-b(s,\bar{Z})] = \left(\frac{b_1(V)-b_1(\bar{V})}{\delta_1}, \frac{b_2(s,X)-b_2(s,\bar{X})}{\delta_2}\right). \]
Therefore
\begin{align*}
\Phi_{\delta_1,\delta_2}'(s) &\lesssim \frac{\bar{L}}{\delta_1}\|b(s,\cdot)-\bar{b}(s,\cdot)\|_{L^{1}(B_{\la})} + \\
& \iint\limits_{B_r\cap G_{\la}\cap \bar{G}_{\la}}  \text{min}\left\lbrace |A^{-1}[b(s,Z)-b(s,\bar{Z})]|,  \frac{1}{\delta_1}\frac{|b_1(V)-b_1(\bar{V})|}{|A^{-1}[Z-\bar{Z}]|}+ \frac{1}{\delta_2}\frac{|b_2(s,X)-b_2(s,\bar{X})|}{|A^{-1}[Z-\bar{Z}]|}  \right\rbrace  dx\,dv  \\
& \leq \frac{\bar{L}}{\delta_1}\|b(s,\cdot)-\bar{b}(s,\cdot)\|_{L^{1}(B_{\la})} + \\
&  \iint\limits_{B_r\cap G_{\la}\cap \bar{G}_{\la}}  \text{min}\left\lbrace |A^{-1}[b(s,Z)-b(s,\bar{Z})]|,  \frac{\delta_2}{\delta_1}\frac{|b_1(V)-b_1(\bar{V})|}{|V-\bar{V}|}+ \frac{\delta_1}{\delta_2}\frac{|b_2(s,X)-b_2(s,\bar{X})|}{|X-\bar{X}|}  \right\rbrace dx\,dv \\
&\leq \frac{\bar{L}}{\delta_1}\|b(s,\cdot)-\bar{b}(s,\cdot)\|_{L^{1}(B_{\la})} + \frac{\delta_2}{\delta_1}\text{Lip}(b_1) \IL^{2N}(B_r)+ \\
&  \iint\limits_{B_r\cap G_{\la}\cap \bar{G}_{\la}} \text{min} \left\lbrace |A^{-1}[b(s,Z)-b(s,\bar{Z})]|,\frac{\delta_1}{\delta_2}\frac{|b_2(s,X)-b_2(s,\bar{X})|}{|X-\bar{X}|}  \right\rbrace dx\,dv \\
& \leq \frac{\bar{L}}{\delta_1}\|b(s,\cdot)-\bar{b}(s,\cdot)\|_{L^{1}(B_{\la})} + \frac{\delta_2}{\delta_1}\text{Lip}(b_1) \IL^{2N}(B_r)+  \iint\limits_{B_r\cap G_{\la}\cap \bar{G}_{\la}}  \Psi(s,z)\, dx\,dv\,,
\end{align*}
where we denoted 
\begin{equation*}
\Psi(s,z)=\text{min} \left\lbrace |A^{-1}[b(s,Z(s,z))-b(s,\bar{Z}(s,z))]|,\frac{\delta_1}{\delta_2}\frac{|b_2(s,X(s,z))-b_2(s,\bar{X}(s,z))|}{|X(s,z)-\bar{X}(s,z)|}  \right\rbrace. 
\end{equation*}

\noindent \textit{Step 3: Definition of the function $\mathfrak{U}$.} Using assumption (R2a), we can now use the estimate of \cite{BC} on the difference quotient of $b_2$,
\begin{equation}\label{Um}
\frac{|b_2(s,X(s,z))-b_2(s,\bar{X}(s,z))|}{|X(s,z)-\bar{X}(s,z)|} \leq \mathfrak{U}(s,X(s,z))+ \mathfrak{U}(s,\bar{X}(s,z)),
\end{equation}
where $\mathfrak{U}$ for fixed $s$ is given by
\begin{equation*}
\mathfrak{U}(s,x)=\sum_{j=1}^{N}\sum_{k=1}^{m}M_j(S_{jk}\textsf{m}_{jk}(s,x)),
\end{equation*}
with $M_j$ a certain smooth maximal operator on $\R^{N}_x$. Inequality \eqref{Um} extends standard maximal inequalities and, due to the presence of singular integrals in the derivative of $b$, requires the use of smooth maximal functions, in which instead of averages one considers averages weighted by a smooth convolution kernel (see Chapter 2 in \cite{Gr} and \cite{BC}).
\medskip

\noindent \textit{Step 4: Estimates on $\Psi$.} Let $\Omega= (t,\tau)\times B_{r}\cap G_{\la}\cap \bar{G}_{\la} \subset \R^{2N+1}$ and $\Omega'= (t,\tau)\times B_{\la} \subset \R^{2N+1}$.
We can estimate the $L^{p}(\Omega)$ norm of $\Psi$ by considering the first element of the minimum and changing variables along the flows:
\begin{equation}\label{psi-Lp}
\|\Psi\|_{L^{p}(\Omega)}\leq \|A^{-1}[b(s,Z)-b(s,\bar{Z})]\|_{L^p(\Omega)}\leq \frac{L+\bar{L}}{\delta_1}\|b\|_{L^{p}(\Omega')}\lesssim_{\la} \frac{1}{\delta_1}.
\end{equation}
Considering now the second element of the minimum and eq.n \eqref{Um}, we can also bound the $\textit{M}\,^{1}(\Omega)$ pseudo-norm of $\Psi$ (where $\textit{M}\,^{p}$, $1\leq p\leq \infty$, is the Lorentz space $L^{p,\infty}$ as defined for instance in Chapter 1 of \cite{Gr1}):
\begin{align*}
|||\Psi|||_{\textit{M}^{1}(\Omega)} & \leq \frac{\delta_1}{\delta_2}|||\mathfrak{U}(s,X)+\mathfrak{U}(s,\bar{X})|||_{\textit{M}^{1}(\Omega)} \leq \frac{\delta_1}{\delta_2}(L+\bar{L})|||\mathfrak{U}|||_{\textit{M}^{1}(\Omega')} \\ 
& \lesssim \frac{\delta_1}{\delta_2}\|\, |||\mathfrak{U}(s,x)|||_{\textit{M}^{1}_{x,v}(B_{\la})} \|_{L^{1}((t,\tau))} \leq \frac{\delta_1}{\delta_2} \parallel |||\mathfrak{U}(s,x)|||_{\textit{M}^{1}(B_{\la}^{x} \times B_{\la}^{v})} \parallel_{L^{1}((t,\tau))}  \\
& \leq \frac{\delta_1}{\delta_2} \parallel \parallel |||\mathfrak{U}(s,x)|||_{\textit{M}^{1}_{x}(B_{\la})} \parallel_{L^{1}_{v}(B_{\la})} \parallel_{L^{1}((t,\tau))} 
\leq \frac{\delta_1}{\delta_2}(2\la)^{N} \|\, |||\mathfrak{U}(s,x)|||_{\textit{M}^{1}(\R^{N})} \|_{L^{1}((t,\tau))},
\end{align*}
that is
\begin{equation*}
|||\Psi|||_{\textit{M}^{1}(\Omega)}\lesssim_{\la}\frac{\delta_1}{\delta_2} \|\, |||\mathfrak{U}(s,x)|||_{\textit{M}^{1}(\R^{N})} \|_{L^{1}((t,\tau))}.
\end{equation*}
From  Theorem 2.10 in \cite{BBC1}, we know 
\[  |||\mathfrak{U}(s,\cdot)|||_{\textit{M}^{1}(\R^{N})} \lesssim \| \textsf{m}(s,\cdot)\|_{\cM(\R^{N})},   \]
and thus
\begin{equation}\label{psi-M1}
|||\Psi|||_{\textit{M}^{1}(\Omega)}\lesssim_{\la}\frac{\delta_1}{\delta_2} \| \textsf{m}\|_{L^{1}((t,\tau);\cM(\R^{N}))}.
\end{equation}
\newline
\medskip
\textit{Step 5: Interpolation.} We have now the ingredients to apply the Interpolation Lemma 2.2 in \cite{BC}, which allows to bound the norm in $L^{1}(\Omega)$ of $\Psi$ using $\|\Psi\|_{L^{p}(\Omega)}$ and $|||\Psi|||_{\textit{M}^{1}(\Omega)}$ as follows:
\begin{equation}
\|\Psi\|_{L^{1}(\Omega)} \lesssim |||\Psi|||_{\textit{M}^{1}(\Omega)}\left[1+\log\left(\frac{\|\Psi\|_{L^{p}(\Omega)}}{|||\Psi|||_{\textit{M}^{1}(\Omega)}} \right)   \right].
\end{equation}
Therefore, using the monotonicity of the functions $\log(y)$ and $y\left[1+\log(\frac{1}{y})\right]$ and the bounds \eqref{psi-Lp} and \eqref{psi-M1}, we get
\begin{equation}\label{interpolation}
\|\Psi\|_{L^{1}(\Omega)} \lesssim_{\la} \frac{\delta_1}{\delta_2}\|\textsf{m}\|\left[1+\log\left(\frac{\delta_2}{\delta_1^{2}\|\textsf{m}\|}  \right)   \right].
\end{equation}\newline
\textit{Step 6: Upper bound for $\Phi_{\delta_1,\delta_2}$.} Integrating in time, from $t$ to $\tau$, the last inequality of Step~2, we obtain
\begin{align}
\Phi_{\delta_1,\delta_2}(\tau)&\lesssim \frac{\bar{L}}{\delta_1}\|b-\bar{b}\|_{L^{1}(\Omega)} + T\frac{\delta_2}{\delta_1}\text{Lip}(b_1) \IL^{2N}(B_r)+ \int_{\Omega} \Psi(s,z)\, dz \,ds \\
&\lesssim \frac{1}{\delta_1}\|b-\bar{b}\|_{L^{1}(\Omega)} + \frac{\delta_2}{\delta_1}+ \|\Psi\|_{L^{1}(\Omega)}.
\end{align}
Therefore, applying \eqref{interpolation} and setting $\frac{\delta_1}{\delta_2}=\alpha$, we get
\begin{equation}\label{boundPhi}
\Phi_{\delta_1,\delta_2}(\tau)\lesssim_{\la} \frac{1}{\delta_1}\|b-\bar{b}\|_{L^{1}(\Omega)} + \frac{1}{\alpha}+\alpha\|\textsf{m}\|\left[1+\log\left(\frac{1}{\delta_1\alpha\|\textsf{m}\|}  \right)   \right].
\end{equation}\newline
\textit{Step 7: Final estimate.} Fix $\gamma>0$. By definition of $\Phi_{\delta_1,\delta_2}$ and $\mu$, since $h$ is non negative, we have
\begin{align}
\Phi_{\delta_1,\delta_2}(\tau) &\geq \int_{B_r\cap\{|Z(\tau,z)-\bar{Z}(\tau,z)|>\gamma\}\cap G_{\la}\cap \bar{G}_{\la}}\log\left(1+\frac{\gamma}{\delta_2}\right)h(s,z)\,dz \nonumber\\
& = \log\left(1+\frac{\gamma}{\delta_2}\right) \mu\left(B_r\cap \{|Z(\tau,z)-\bar{Z}(\tau,z)|>\gamma\}\cap G_{\la}\cap \bar{G}_{\la} \right).\nonumber
\end{align}
This implies that
\begin{equation}\label{mis}
\mu\left(B_r\cap \{|Z(\tau,z)-\bar{Z}(\tau,z)|>\gamma\} \right)\leq \frac{\Phi_{\delta_1,\delta_2}(\tau)}{\log\left(1+\frac{\gamma}{\delta_2}\right)}+\mu(B_r\setminus G_{\la})+\mu(B_r\setminus \bar{G}_{\la}).
\end{equation}
Combining \eqref{boundPhi} and \eqref{mis} we obtain
\begin{align*}
& \mu\left(B_r\cap \{|Z(\tau,z)-\bar{Z}(\tau,z)|>\gamma\} \right) \\
&\leq C_{\la}\left\lbrace \frac{\frac{\|b-\bar{b}\|}{\delta_1}+\frac{1}{\alpha}+\alpha\|\textsf{m}\|\left[1+\log\left(\frac{1}{\delta_1\alpha\|\textsf{m}\|}\right)\right]}{\log\left(1+\frac{\gamma}{\delta_2}\right)}\right\rbrace 
+\mu(B_r\setminus G_{\la})+\mu(B_r\setminus \bar{G}_{\la})\\
& = C_{\la}\left\lbrace\frac{\|b-\bar{b}\|}{\delta_1 \log\left(1+\frac{\gamma}{\delta_2}\right)}+\frac{1}{\alpha \log\left(1+\frac{\gamma}{\delta_2}\right)}+ \frac{\alpha\|\textsf{m}\|\left[1+\log\left(\frac{1}{\delta_1\alpha\|\textsf{m}\|}\right)\right]}{\log\left(1+\frac{\gamma}{\delta_2}\right)} \right\rbrace +\mu(B_r\setminus G_{\la})+\mu(B_r\setminus \bar{G}_{\la}) \\
& = 1)+2)+3)+4)+5).
\end{align*}
Fix $\eta >0$. Since $b$ and $\bar{b}$ satisfy assumption (R1a), we can choose $\la >0$ large enough so that $4)+5)\leq \frac{2\eta}{4}$. Then, replacing $\delta_1$ with $\alpha\cdot\delta_2$, we notice that $3)$ is uniformly bounded for $\delta_2\rightarrow 0$, so we can choose $\alpha$ small enough in order to get $3)\leq \frac{\eta}{4}$. Now $\la$ and $\alpha$ are fixed, but $\delta_1$ and $\delta_2$ are free to be chosen as long as the ratio equals $\alpha$. Hence we choose $\delta_2$ small enough so that $2)\leq \frac{\eta}{4}$. This fixes all parameters.

Setting
\[ C_{\gamma,r,\eta}=\frac{C_{\la}}{\delta_1\log\left(1+\frac{\gamma}{\delta_2}\right)}  \]
we have proven our statement.
\end{proof}

\subsection{Uniqueness, stability and compactness}

In this subsection we use the result obtained in Theorem \ref{teo2} to show uniqueness, stability, and compactness of the regular Lagrangian flow.

\begin{corollary}[Uniqueness]
Let $b$ be a vector field satisfying assumptions (R1a), (R2a) and (R3), and fix $t\in [0,T]$. Then, the $\mu$-regular Lagrangian flow associated with $b$ starting at time $t$, if it exists, is unique $\mu$-a.e..
\end{corollary}

\begin{proof}
Let $Z$ and $\bar{Z}$ be two $\mu$-regular Lagrangian flows associated with the same vector field $b$. Then from Theorem \ref{teo2}, setting $b=\bar{b}$, we have
\begin{equation}
\mu (B_r \cap \{|Z(s,\cdot)-\bar{Z}(s,\cdot)|>\gamma\})\leq \eta,
\end{equation}
for all $\gamma,\,r,\,\eta>0$ and for all $s\in [0,T]$. This implies $Z=\bar{Z}$ $\mu$-a.e..
\end{proof}

\begin{corollary}[Stability]
Let $\{b_n\}$ be a sequence of vector fields satisfying assumption (R1a), converging in $L^{1}_{\rm loc}([0,T]\times \R^{2N})$ to a vector field $b$ which satisfies assumptions (R1a), (R2a) and (R3). Assume that there exist $Z_n$ and $Z$ $\mu$-regular Lagrangian flows starting at time $t$ associated with $b_n$ and $b$ respectively, and denote by $L_n$ and $L$ the compressibility constants of the flows. Suppose that:
\begin{itemize}
\item The measure of the superlevels associated with $Z_n$ in hypothesis (R1a) is bounded by some functions $g_n(r,\la)$ which go to zero uniformly in $n$ as $\la \rightarrow \infty$ at fixed $r$,
\item The sequence $\{L_n\}$ is equi-bounded.
\end{itemize}
Then the sequence $\{Z_n\}$ converges to $Z$ locally in measure with respect to $\mu$ in $\R^{2N}$, uniformly in $s$ and $t$.
\end{corollary}

\begin{proof}
We set $\bar{b}=b_n$ and $\bar{Z}=Z_n$ in Theorem \ref{teo2}, then there exist two positive constants $\la$ and $C_{\gamma, r, \eta}$, which are independent of $n$, such that for all $s\in [0,T]$ it holds
\begin{equation*}
\mu(B_r \cap \{|Z(s,\cdot)-Z_n(s,\cdot)|>\gamma\})\leq C_{\gamma, r, \eta} \|b-b_n\|_{L^{1}((0,T)\times B_{\la})} + \eta.
\end{equation*}
In particular, for any $r,\,\gamma >0$ and any $\eta >0$, we can choose $\bar{n}$ large enough so that  
\begin{equation*}
\mu(B_r \cap \{|Z(s,\cdot)-Z_n(s,\cdot)|>\gamma\})\leq 2 \eta \quad \quad \mbox{for all $n\geq \bar{n}$ and $s\in [t,T]$},
\end{equation*}
which is the thesis.
\end{proof}

\begin{corollary}[Compactness]\label{cor:compactness-flow}
Let $\{b_n\}$ be a sequence of vector fields satisfying assumptions (R1a), (R2a) and (R3), converging in $L^{1}_{\rm loc}([0,T]\times \R^{2N})$ to a vector field $b$ which satisfies assumptions (R1a), (R2a) and (R3). Assume that there exist $Z_n$ $\mu$-regular Lagrangian flows starting at time $t$ associated with $b_n$. Suppose that:
\begin{itemize}
\item The measure of the superlevels associated with $Z_n$ in hypothesis (R1a) is bounded by some functions $g_n(r,\la)$ which go to zero uniformly in $n$ as $\la \rightarrow \infty$ at fixed $r$,
\item For any compact subset $K$ of $\R^{2N}$,
\begin{equation}\label{loglog}
\int_{K}\log(1+\log(1+|Z_n(s,z)|))\,d\mu(z)
\end{equation}
is equi-bounded in $n$ and $s,t$,
\item For some $p>1$ the norms $\|b_n\|_{L^{p}((0,T)\times B_r)}$ are equi-bounded for any fixed $r>0$,
\item The norms of the singular integral operators associated with the vector fields $b_n$ (as well as their number $m$) are equi-bounded,
\item The norms of $\textsf{m}_{jk}^{n}$ in $L^{1}((0,T);\mathcal{M}( \R^{N}))$ are equi-bounded in $n$.
\end{itemize} 
Then as $n\rightarrow \infty$ the sequence $\{Z_n\}$ converges to some $Z$ locally in measure with respect to $\mu$, uniformly with respect to $s$ and $t$, and $Z$ is a regular Lagrangian flow starting at time $t$ associated with $b$.
\end{corollary}

\begin{proof}
We apply Theorem \ref{teo2} with $b=b_n$ and $\bar{b}=b_m$. Observe that the compressibility constants $L$ and $\bar{L}$ of the same theorem are equal to $1$. Indeed $b$ and $\bar{b}$ are divergence free as they both satisfy assumption (R2a). Hence we have for any $r,\gamma>0$ 
\begin{equation*}
\mu(B_r \cap \{|Z_n(s,\cdot)-Z_m(s,\cdot)|>\gamma\})\rightarrow 0 \quad \mbox{as $m,n\to \infty$, uniformly in $s,t$.}
\end{equation*}
Thus it follows that $Z_n$ converges to some $Z\in C([t,T];L^{0}_{\rm loc}(\R^{2N},d\mu))$ locally in measure with respect to $\mu$, uniformly in $s,t$. The uniformity in $n$ and $s,t$ of the bound \eqref{loglog} implies $Z\in \mathcal{B}([t,T];\log\log L_{\rm loc}(\R^{2N},d\mu)).$
We notice that conditions $(2)$ and $(3)$ in Definition \ref{def:RLF} are satisfied, since thanks to (R2a) the vector fields $b_n$ are divergence free. We are left with the proof of condition $(1)$. Observe that a $\beta \in C^{1}(\R^{2N})$ can be approximated by a sequence of $\beta_{\epsilon} \in C^{1}_c(\R^{2N})$, therefore it suffices to show condition $(1)$ for this latter class of functions. To this end we want to perform the limit in $n$ of equation \eqref{eq:RSder} written for $Z_n$ and $b_n$. From the convergence in measure of $Z_n$ to $Z$ and the fact that $\beta_\e$ is compactly supported, it follows the convergence in distributional sense of $\beta_{\epsilon}(Z_n)$ to $\beta_{\epsilon}(Z)$ and of $\beta'_{\epsilon}(Z_n)$ to $\beta'_{\epsilon}(Z)$. While using the uniform bound of $\|b_n\|_{L^{p}((0,T)\times B_r)}$ and Lusin's Theorem, we get convergence in $L^{1}_{\rm loc}$ of $b_n(Z_n)$ to $b(Z)$. Thus we have convergence in the sense of distribution to equation \eqref{eq:RSder}. 
\end{proof}

 The above compactness statement does not directly translate into an existence result for Lagrangian flows, since in general it is not trivial to find a sequence $b_n$ approximating $b$ as in the hypotheses of Corollary \ref{cor:compactness-flow}. This is due to the fact that the function $g(r,\lambda)$ in Lemma \ref{estsuper} does not depend only on bounds on the vector field, but also on bounds on the density of charge. We are able to do this in the specific case of the flow associated with the Vlasov-Poisson equation (solution to \eqref{eq:charXV}) and therefore we postpone this to Section \ref{sec:4}.


\section{Useful estimates}\label{sec:2}

In this Section we recall some well known a priori estimates on physical quantities related to the Vlasov-Poisson equation and we adapt them to the context of the system \eqref{eq:VP-dirac}-\eqref{eq:newton}.

\begin{proposition}\label{prop:Sobolev}
Let $\rho(t,\cdot)\in L^s(\R^3)$, for some $s$ such that $1 \leq s\leq\infty$. Then
\begin{eqnarray}
\|E(t,\cdot)\|_{L^{3s/(3-s)}}\leq C \|\rho(t,\cdot)\|_{L^s}\,, && \mbox{ if } s\in(1,3)\,,\label{eq:s3}\\
\nonumber\\
\|E(t,\cdot)\|_{{C}^{0,\alpha}}\leq C\|\rho(t,\cdot)\|_{L^s}\,, && \mbox{ if } s> 3\,, \mbox{ with } \alpha=1-\frac{3}{s}\,,  
\label{eq:morrey}\\
\nonumber\\
|||E(t,\cdot)|||_{M^{3/2}}\leq C\|\rho(t,\cdot)\|_{L^1}\,,\label{eq:s1}
\end{eqnarray}
where $C$ is a constant depending only on $s$. 
\end{proposition}
\begin{proof}
We observe that the electric field can be written as $E(t,x)=4\pi\nabla_x\Delta_x^{-1}\rho(t,x)$. Eq.ns \eqref{eq:s3} and \eqref{eq:morrey} easily follow respectively from Gagliardo--Nirenberg--Sobolev and Morrey inequalities in dimension three (see for instance \cite{E}). Inequality \eqref{eq:s1} is a direct consequence of Hardy-Littlewood-Sobolev inequality. 
\end{proof}

\begin{proposition}[Mass and energy conservation]\label{prop:energy}
Let 
\begin{align*}
M(t)&=\iint f(t,x,v)\,dxdv\,,
\\
H(t)&=\iint \frac{|v|^2}{2}f(t,x,v)dxdv+ \frac{|\eta(t)|^2}{2}+ \frac{1}{2}\iint \frac{\rho(t,x)\rho(t,y)}{|x-y|}dxdy+\int \frac{\rho(t,x)}{|x-\xi(t)|}dx\,,
\end{align*}
be respectively the total mass and the total energy associated with the system \eqref{eq:VP-dirac}-\eqref{eq:newton}. If $f(t)$ and $\xi(t)$ are solutions to \eqref{eq:VP-dirac}-\eqref{eq:newton} on $[0,T]$, then $M(t)$ and $H(t)$ are conserved quantities w.r.t. time. 
\end{proposition}
\begin{proof}
It follows from direct inspection by performing the time derivative of $M(t)$ and $H(t)$. \end{proof}


As a consequence of Proposition \ref{prop:energy}, we observe that if the energy $H(t)$ is assumed to be initially finite, then it is bounded for all times. This ensures in particular that the velocity of the Dirac mass located at $\xi(t)$ is finite. 

\begin{proposition}\label{prop:charge-bounds}
Let $T>0$ such that for all $t\in[0,T]$, $f(t)$ and $\xi(t)$ are solutions of the system \eqref{eq:VP-dirac}-\eqref{eq:newton} with finite associated initial energy $H(0)$. Then 
\begin{eqnarray}
&|\xi(t)|\leq|\xi_0|+T\sqrt{2H(0)}\,,\label{eq:xi}\\
\nonumber\\
&|\eta(t)|\leq \sqrt{2H(0)}\,.\label{eq:eta}
\end{eqnarray}
\end{proposition}
\begin{proof}
We observe that $H(t)$ is a sum of positive terms. Notice that here we are heavily using the electrostatic nature of the particles in the plasma. In the gravitational case, the total energy has a nonpositive term. By Proposition \ref{prop:energy}, $H(t)=H(0)$ is finite, hence 
\begin{equation*}
\frac{|\eta(t)|^2}{2}\leq H(0)\,,
\end{equation*} 
from which estimate \eqref{eq:eta} easily follows. We can use this bound in the first equation of \eqref{eq:newton} to get
\begin{equation*}
|\xi(t)|\leq |\xi_0|+\int_0^t|\eta(s)|ds
\end{equation*}
which leads to \eqref{eq:xi} when using \eqref{eq:eta} and then taking the supremum in $t\in[0,T]$.
\end{proof}

\begin{proposition}[Proposition 2.1  in \cite{DMS}]\label{prop:interpolation}
Let $m\geq 0$, $f(t,\cdot,\cdot)\in L^1(\R^3\times \R^3)$ and $\rho(t,\cdot)\in L^{1}(\R^3)$ as in \eqref{eq:VP-dirac}. Then there exists a constant $C>0$, which only depends on $m$, such that
\begin{equation}\label{eq:interpolation}
\|\rho(t,\cdot)\|_{L^{\frac{m+3}{3}}}\leq C\|f(t,\cdot,\cdot)\|_{L^\infty}^{\frac{m}{m+3}}\left(\iint |v|^m f(t,x,v)\,dx\,dv\right)^{\frac{3}{m+3}}\,.
\end{equation}
\end{proposition}

\begin{proposition}\label{prop:rho}
Let $f\geq 0$, $f(t,\cdot,\cdot)\in L^1\cap L^\infty(\R^3\times\R^3)$ 
solution to \eqref{eq:VP-dirac}. Assume the total energy  to be initially finite, then $\rho(t,\cdot)\in L^1\cap L^{5/3}(\R^3)$ and $E(t,\cdot)\in L^q(\R^3)$, for any $\frac{3}{2}<q\leq \frac{15}{4}$. 
\end{proposition}
\begin{proof}
The bound $\rho(t,\cdot)\in L^{5/3}(\R^3)$ follows by Proposition \ref{prop:interpolation} for $m=2$. The estimate on the electric field is a consequence of Proposition  \ref{prop:Sobolev} for $s=1$ and $s=\frac{5}{3}$.
\end{proof}

The following two propositions regard specifically the case in which we deal with a Dirac mass and their proof relies on the condition that the total charge $M(0)$ has to be strictly less than one. This is the only reason why we need to assume \eqref{eq:M(0)} in Theorem \ref{thm:main}.

\begin{proposition}[Proposition 2.9 in \cite{DMS}]\label{prop:viriale}
Let $M(0)<1$, $H(0)<+\infty$ and $(f,\xi)$ a classical solution to \eqref{eq:VP-dirac}-\eqref{eq:newton} on $[0,T]$. Then for all $t\in [0,T]$ there is a constant depending only on $M(0)$ and $H(0)$ such that
\begin{equation}\label{eq:viriale}
\int_0^t \iint \frac{f(s,x,v)}{|x-\xi(s)|^2}\,dx\,dv\,ds\leq C(1+t)\,.
\end{equation}
\end{proposition}

\begin{proposition}[Theorem 1.1. in \cite{DMS}]\label{prop:moments}
Let $f_0\in L^1\cap L^\infty(\R^3\times\R^3)$ non-negative, $(\xi_0,\eta_0)\in\R^3\times\R^3$ and $H(0)$ finite. Assume further that 
\begin{itemize}
\item[\rm (i)] $M(0)<1$,
\item[\rm (ii)] There exists $m_0>6$ such that for all $m<m_0$
$$\iint \left(|v|^2+\frac{1}{|x-\xi_0|}\right)^{m/2}f_0(x,v)\,dx\,dv<+\infty\,.$$
\end{itemize} 
Then there exists a global weak solution $(f,\xi)$ to the system \eqref{eq:VP-dirac}-\eqref{eq:newton}, with\break $f\in{C}(\R_+,L^p(\R^3\times\R^3))\cap L^\infty(\R_+,L^\infty(\R^3\times\R^3))$ for any $1\leq p<+\infty$, $\xi\in{C}^2(\R_+)$, and $E\in L^{\infty}([0,T],{C}^{0,\alpha}(\R^3))$ for all $T>0$. 

Moreover, for all $t\in\R_+$ and for all $m<{\rm min}(m_0,7)$, 
\begin{equation}\label{eq:momentsbis}
\iint \left(|v|^2+\frac{1}{|x-\xi(t)|}\right)^{m/2}f(t,x,v)\,dx\,dv\leq C(1+t)^c\,,
\end{equation}
where $C$ and $c$ only depend on the initial data.
\end{proposition}

\begin{remark}\label{remark:rho}
Observe that thanks to Proposition \ref{prop:interpolation}, condition \eqref{eq:momentsbis} implies $\rho(t)\in L^s(\R^3)$, for $s>3$. Hence the H\"{o}lder continuity of the electric field follows directly by Proposition~\ref{prop:Sobolev}.
\end{remark}


\section{Proof of Theorem \ref{thm:main}}\label{sec:4}

\subsection{Existence of the Lagrangian flow}\label{subsec:flow-existence}

In this subsection we shall use the results obtained in Section 2 for a general flow solution to equation \eqref{eq:characteristicsystem} and apply them to the context of the Vlasov-Poisson system \eqref{eq:VP-dirac}-\eqref{eq:newton}, namely to the ODE \eqref{eq:charXV}. In particular we will prove existence of a flow associated with the vector field $b(s,x,v)=(v,E(s,x)+F(s,x))$, using the compactness result provided by Corollary~\ref{cor:compactness-flow}. To this end it suffices to construct a sequence $b_n$ which approximates $b$ and satisfies the hypotheses of Corollary \ref{cor:compactness-flow}.
\medskip

Let $f_0$ and $(\xi_0,\eta_0)$ be the initial data of system \eqref{eq:VP-dirac}, satisfying the hypotheses of Theorem~\ref{thm:main}. We consider the approximating initial densities given by 
\begin{equation}\label{eq:indata-n}
f_0^{n}(x,v)= f_0(x,v) \mathbbm{1}_{\{(x,v): \frac{1}{n}<|x-\xi_0|<n,\, |v-\eta_0|<n\}}(x,v).
\end{equation}
Thanks to \cite{MMP}, this choice ensures existence and uniqueness of $f_n$ and $(\xi_n,\eta_n)$, solutions to the Vlasov-Poisson system \eqref{eq:VP-dirac}-\eqref{eq:newton}. Moreover $f_n$ is a Lagrangian solution, i.e.
\begin{equation}\label{eq:pushf}
f_n(s,X_n(s,x,v),V_n(s,x,v))=f_0^{n}(x,v),
\end{equation}
where $(X_n,V_n)$ satisfy
\begin{equation}
\left\{ 
\begin{array}{l}
\dot{X}_n(s,x,v)=V_n(s,x,v)\\
\dot{V}_n(s,x,v)=E_n(s,X_n(s,x,v))+F_n(s,X_n(s,x,v)),
\end{array}\right.
\end{equation}
with 
\begin{equation}
\begin{array}{l}
E_n(s,x)= \left(\nabla \frac{1}{|\cdot |}* \rho_n\right)(s,x),\quad 
\rho_n(s,x)=\int f_n(s,x,v)\, dv,\quad
F_n(s,x)=\frac{x-\xi_n(s)}{|x-\xi_n(s)|^{3}}.
\end{array}
\end{equation}

From now on the abstract measure $\mu$ of Section 2 will be set as $\mu= f_0\, \mathcal{L}^{2N}$, where $f_0$ is the initial density of our problem. In order to apply Corollary \ref{cor:compactness-flow}, we need then the approximating vector fields $b_n(s,x,v)=(v,E_n(s,x)+F_n(s,x))$ to satisfy hypotheses (R1a), (R2a), and (R3) ``uniformly" in $n$ (with equi-bounds on the quantities involved) and the bound \eqref{loglog}. Furthermore we set the dimension $N$ equal to $3$.
\bigskip

\noindent\textbf{Proof of (R1a) + equibound: control of superlevels}
\medskip

\noindent In \cite{BBC1} a control on the superlevels was obtained using hypothesis (R1) which provided an upper bound on the integral of $\log(1+|Z|)$. Without assumption (R1), we need estimates on $|V|^{2}$ in order to control the superlevels. This requires integrating a function which grows slower than $\log(1+|V|)$ at infinity. Furthermore, differently from \cite{BBC2}, we will bound the superlevels of $Z$ with respect to the measure $\mu=f_0\, \mathcal{L}^{6}$. For the sake of clarity we will use the notation $f_0(B)$ to indicate the measure $\mu$ of a set $B\subseteq \IR^6$. The result is the following lemma.

\begin{lemma}\label{estsuper}
Let $b(t,x,v)=(v,E(t,x)+F(t,x))$  and let $Z:[t,T]\times\IR^3\times\IR^3\rightarrow\IR^3\times\IR^3$ be the $\mu$-regular Lagrangian flow relative to $b$ starting at time $t$, with sublevel $G_\lambda$. Assume $M(0) < 1$.  Then, for all $r,\lambda>0$, we have
\begin{align*}
f_0(B_r \setminus G_\lambda)\leq g(r,\lambda),
\end{align*}
where the function $g$ depends only on $\|E\|_{L_t^{\infty}(L_x^{{5}/{2}})}$, $\|f\|_{L_t^\infty(L_{x,v}^{\infty})}$, $M(0)$, $H(0)$, and $g(r,\lambda)\downarrow 0$ for $r$ fixed and $\lambda\uparrow \infty$. 
\end{lemma}
\begin{remark}
Notice that this lemma holds also for the regularized problem (system \eqref{eq:VP-dirac}-\eqref{eq:newton} with initial density $f_0^{n}$). Therefore we have, for all $r,\la >0$, 
\be
f_{0}^{n}(B_r \backslash G_{\lambda}^{n}) \leq g_n(r,\la),
\ee
where $g_n$ converges to zero for $r$ fixed and $\la \uparrow \infty$. Moreover, this convergence is uniform in~$n$. Indeed the proof of Lemma \ref{estsuper} entails the functions $g_n$ to be increasing with respect to the norms of $E_n$, $U_n$, $F_n$, $f_n$, and with respect to $H_n(0)$. These quantities are in turn all bounded by the same quantities without the index $n$. Therefore, due to the choice of the initial densities of the regularized problem, we have
\be
\begin{split}
f_{0}(B_r \backslash G_{\lambda}^{n})&\leq f_{0}^{n}(B_r \backslash G_{\lambda}^{n})+f_0 \Big(\R^{6}\backslash \big\{(x,v):\frac{1}{n}<|x-\xi_0|<n,|v-\eta_0|<n \big\}\Big)\\
&\leq g_n(r,\la)+f_0 \Big( \big\{(x,v):|x-\xi_0|\leq \frac{1}{n}\, \text{or}\, |x-\xi_0|\geq n \big\} \Big)\\
&\ \ \ \ +f_0 \big(\big\{(x,v):|v-\eta_0|\geq n\big\}\big),
\end{split}
\ee
where $g_n(r,\la)$ depends on the norms of $E$, $U$, $F$, $f$ and on $H(0)$, and tends to zero as $\la \rightarrow \infty$ uniformly in $n$. Moreover the last two terms tend to zero as $n\rightarrow \infty$ by Lebesgue's Dominate Convergence Theorem. Hence we have, for any fixed $\epsilon,\, r>0$, that there exist $\la>0$ and $N \in \N$ such that
\begin{equation}
f_{0}(B_r \backslash G_{\lambda}^{n})\leq \epsilon
\end{equation} 
for each $n\geq N$.
\end{remark}

\begin{proof}[Proof of Lemma \ref{estsuper}]
We call $\tilde{G}_{\lambda}$ the sublevel of $V$ and we remark that by the first equation in \eqref{eq:characteristicsystem}, whenever $(x,v)\in \tilde{G}_{\lambda}$ one has $|X(s,x,v)|\leq |x|+|s-t|\lambda$ and $|Z(s,x,v)|\leq |x|+(1+T)\lambda$. Thus for $\lambda > r$ one has $B_r \backslash G_{\lambda}\subset B_r \backslash \tilde{G}_{(\lambda - r)/(1+T)}$, while for $\lambda \leq r$ we can just use that $B_r\backslash G_{\lambda} \subset B_r$, so to conclude the proof it suffices to bound the superlevels of $V$.
In order to do this we will first prove that
\begin{equation}\label{sup1+logv2<cost}
\iint_{B_r} \sup_{s\in [t,T]}\log\left(1+\log\left(1+\frac{|V(s,x,v)|^{2}}{2}\right)\right)\,f_0(x,v)\,dx\,dv \leq A
\end{equation}
where $A$ is a constant depending on $\|E\|_{L^{\infty}_t(L^{5/2}_x)}$, $\|f\|_{L^{\infty}_t(L_{x,v}^{\infty})}$, $M(0)$ and $H(0)$. Once one has shown that \eqref{sup1+logv2<cost} holds, the statement of the lemma follows simply by the following inequality:
\begin{align}\label{eq:Glambdabound}
\iint_{B_r} \sup_{s\in [t,T]}\log\left(1+\log\left(1+\frac{|V(s,x,v)|^{2}}{2}\right)\right)\,f_0(x,v)\,dx\,dv\nonumber \\
\geq f_0(B_r\setminus \tilde{G}_\lambda) \log\left(1+\log\left(1+\frac{\lambda^{2}}{2}\right)\right).
\end{align}
Consider the ODE system \eqref{eq:charXV} and recall the Definition \ref{def:RLF}. Let $\beta(z)=\log\left(1+\log\left(1+\frac{|z|^{2}}{2}\right)\right)$, then
 \begin{align}\label{eq:betaprime}
 &\beta'(z)=\frac{z}{\left(1+\log\left(1+\frac{ {|z|^2}}{2}\right)\right)\left( {1+\frac{ {|z|^2}}{2}}\right)}\,.
  \end{align}
 Using \eqref{eq:RSder} and \eqref{eq:betaprime}, we compute
 \begin{equation*}
 \p_s[\beta(V(s,x,v))]= \frac{(E(s,X(s,x,v))+F(s,X(s,x,v)))\cdot V(s,x,v) }{\left(1+\log\left(1+\frac{|V(s,x,v)|^2}{2}\right)\right)\left(1+\frac{|V(s,x,v)|^2}{2}\right)}=\Phi_1(s,x,v)+\Phi_2(s,x,v)\,. 
 \end{equation*}
By integrating the above equation w.r.t. time we get
 \begin{equation}\label{int}
  \log\left(1+\log\left(1+\frac{ {|V(s,x,v)|^2}}{2}\right)\right)= \log\left(1+\log\left(1+\frac{ {|v|^2}}{2}\right)\right)+\int_{t}^{s}\left(\Phi_1(\tau,x,v)+\Phi_2(\tau,x,v)\right)d\tau.
 \end{equation} 
Then we plug \eqref{int} in \eqref{sup1+logv2<cost}, verifying that we obtain an upper bound. The inequality
\begin{equation*}
\log\left(1+\log\left(1+\frac{|v|}{2}^2\right)\right)\leq \frac{|v|}{2}^2
\end{equation*}
allows to estimate the first term multiplied by $f_0$ with the kinetic energy, that is in turn bounded by the initial total energy $H(0)$. Thus this term belongs to $L^{1}(\IR_x^{3}\times \IR_v^{3}, d\mu)$ with $\mu=f_0\,\mathcal{L}^{6}$, so it remains to estimates the terms in the time integral. For the first term we have
\begin{equation}
\begin{split}
&\iint_{B_r} \sup_{s\in [t,T]}\int_t^{s}\Phi_1(\tau,x,v)\,f_0(x,v)\,d\tau \,dx\,dv \leq \iint_{B_r}\int_0^{T}|\Phi_1(\tau,x,v)|\,f_0(x,v)\,d\tau \,dx\,dv\\
& \leq \int_0^{T}\iint_{B_r}\frac{|E(\tau,X(\tau,x,v))|| V(\tau,x,v)| }{\left(1+\log\left(1+\frac{|V(\tau,x,v)|^2}{2}\right)\right)\left(1+\frac{|V(\tau,x,v)|^2}{2}\right)}f_0(x,v)\,dx\,dv\, d\tau\\
& = \int_0^{T}\iint_{B_r}\frac{|E(\tau,X(\tau,x,v))|| V(\tau,x,v)| }{\left(1+\log\left(1+\frac{|V(\tau,x,v)|^2}{2}\right)\right)\left(1+\frac{|V(\tau,x,v)|^2}{2}\right)}f(\tau,X(\tau,x,v),V(\tau,x,v))\,dx\,dv\, d\tau\\
& = \int_0^{T}\iint_{B_r}\frac{|E(\tau,x)||v| }{\left(1+\log\left(1+\frac{|v|^2}{2}\right)\right)\left(1+\frac{|v|^2}{2}\right)}f(\tau,x,v)\,dx\,dv\, d\tau\\
& \leq C \int_0^{T}\iint_{B_r}|E(\tau,x)|f(\tau,x,v)\,dx\,dv\, d\tau \leq C \int_0^{T}\int|E(\tau,x)|\rho(\tau,x)\,dx\,d\tau \\
& \leq C \int_0^{T}\|E\|_{L^{5/2}}\|\rho\|_{L^{5/3}} \leq C\, T \|E\|_{L^{\infty}(L^{5/2})}\|\rho\|_{L^{\infty}(L^{5/3})}
\end{split}
\end{equation}

For the last term of \eqref{int} we compute

\begin{equation}\label{eq:phi4}
\begin{split}
\iint_{B_r}& \sup_{s\in [t,T]}\int_t^{s} \Phi_{2}(\tau,x,v)\, f_0(x,v),d\tau\, dx\, dv \leq \int_0^{T}\iint |\Phi_{2}(\tau,x,v)|\, f_0(x,v)\, dx\, dv\,d\tau \\
& = \int_0^{T}\iint \frac{|F(\tau,X(\tau,x,v)) \. V(\tau,x,v)| }{\left(1+\log\left(1+\frac{|V(\tau,x,v)|^2}{2}\right)\right)\left(1+\frac{|V(\tau,x,v)|^2}{2}\right)} \, f_0(x,v) \, dx\, dv\,d\tau\,. 
\end{split}
\end{equation}
Since the denominator of the integrand is bounded from below, we can estimate the above quantity as follows:   
\begin{equation*}
\begin{split}
\eqref{eq:phi4}
& \leq C \int_0^{T}\iint |F(\tau,X(\tau,x,v))|\,f_0(x,v)\,dx\,dv\,d\tau \\
& = C \int_0^{T}\iint \frac{f_0(x,v)}{|X(\tau,x,v)-\xi(\tau)|^{2}} dx\, dv\,d\tau \\
& = C \int_0^{T}\iint \frac{f(\tau,X(\tau,x,v),V(\tau,x,v))}{|X(\tau,x,v)-\xi(\tau)|^{2}} dx\, dv\,d\tau \\
& = C \int_0^{T}\iint \frac{f(\tau,x,v)}{|x-\xi(\tau)|^{2}} dx\, dv\,d\tau \leq C(1+T) , 
\end{split}
\end{equation*}
where in the last inequality we used Proposition \ref{prop:viriale}.

Thus, condition \eqref{sup1+logv2<cost} is satisfied and the proof is completed thanks to \eqref{eq:Glambdabound}.
\end{proof}

\bigskip

\noindent\textbf{Proof of (R2a): spatial regularity}
\medskip

\noindent Since $b_n(t,x,v)=(b_1^{n}(v),b_2^{n}(t,x))$ with $b_1^{n}(v)=v$ and $b_2^{n}(t,x)=E_n(t,x)+F_n(t,x)$, we observe that the Lipschitz constants of $b_1^{n}$ and $b_1$ are trivially equi-bounded. We are left to show that the derivatives of $b_2^{n}$ and $b_2$ are singular integrals of fundamental type on $\R^{3}$ of finite measures, and that the norms of the kernels associated with the singular integral operators and those of the measures in $L^{1}((0,T);\mathcal{M}(\R^{3}))$ are equi-bounded. We compute, outside of the origin,

\begin{align*}
\partial_{x_j}(b_2)_i(x)=  \partial_{x_j}(E+F)_i(x)&= \partial_{x_j}\left( \frac{\cdot}{|\cdot |^{3}} \ast \rho(t,\cdot) \right)_i (x)+ \partial_{x_j}\left( \frac{\cdot}{|\cdot |^{3}} \ast \delta_{\xi(t)} \right)_i (x)\\
&= \left( \partial_{x_j}\frac{(\cdot)_i}{|\cdot |^{3}} \ast \left(\rho(t,\cdot)+\delta_{\xi(t)}\right) \right) (x)\\
&= \left( \frac{\delta_{ij}|\cdot|^{2}-3\cdot_i \cdot_j}{|\cdot|^{5}} \ast \left(\rho(t,\cdot)+\delta_{\xi(t)}\right) \right) (x).
\end{align*}
Therefore $\partial_{x_j}(b_2)_i$ is a singular integral of the finite measure $\rho+\delta_{\xi(t)}$, with kernel
\[ K_{ij}(y)= \frac{\delta_{ij}|y|^{2}-3 y_i y_j}{|y|^{5}}. \]
The kernel satisfies conditions of Def.2.13 in \cite{BC}, therefore it is a singular kernel of fundamental type. Similarly we have $ \partial_{x_j}(b_2^{n})_i= K_{ij}(\cdot)\ast (\rho(t,\cdot)+\delta_{\xi_n(t)})$, hence also $\partial_{x_j}(b_2^{n})_i$  are singular integrals of finite measures, with equi-bounded kernels and equi-bounds on the measures' norms.


\bigskip

\noindent\textbf{Proof of (R3)}
\medskip

\noindent We shall prove now that the $L^{p}$-norms of $b$ and $b_n$ in $(0,T)\times B_r$ are equi-bounded, for some $p>1$ and for any fixed $r>0$. Through an easy computation we notice that the $M^{3/2}$-pseudo-norms of $F$ and $F_n$ are equi-bounded and uniform in $t$:
$$
|||F_n(t,\cdot)|||^{{3/2}}_{M^{3/2}} =\sup_{\la >0} \left\lbrace \la^{3/2}\mathcal{L}^{3}\left(\{x:\frac{1}{|x-\xi_n(t)|^{2}}>\la\}\right)\right\rbrace=
\sup_{\la >0}\left\{\la^{3/2}\int_{|x-\xi_n(t)|<\frac{1}{\sqrt{\la}}} 1\, dx \right\} \leq C \,.
$$
Similarly we have that the $L^{1}$-norms of $F$ and $F_n$ are equi-bounded in $(0,T)\times B_r$ for any $r>0$:
$$
\sup_{t\in[0,T]}\|F_n\|_{L^{1}( B_r)} =\sup_{t\in[0,T]}\int_{B_r}\frac{1}{|x-\xi_n(t)|^{2}} dx= \sup_{t\in[0,T]}\int_{B_r(\xi_n(t))}\frac{1}{|y|^{2}} dy \leq C \,.
$$
Furthermore Proposition \ref{prop:Sobolev} tells us that $E$ and $E_n$ belong to $L^{\infty}((0,T);M^{3/2}(\R^{3}))$, with the respective pseudo-norms which are equi-bounded in $n$. Therefore the second component of the vector fields $b$ and $b_n$ (i.e. $E+F$, $E_n+F_n$) are equi-bounded in the space $L^{\infty}((0,T);M^{3/2}(\R^{3}))\subset L^{p}_{\rm loc}((0,T)\times \R^{3})$ for any $1\leq p <\frac{3}{2}$. Since $v\in L^{p}_{\rm loc}((0,T)\times \R^{3})$ for any $p$, we conclude that $b$, $b_n$ belong to $L^{p}_{\rm loc}((0,T)\times \R^{3})$ for any $1\leq p <\frac{3}{2}$, with uniform bound on the norms.
\bigskip

\noindent\textbf{Proof of the equi-boundedness of \eqref{loglog}}
\medskip

\noindent We observe that 
\begin{equation}
|Z_n|\leq |X_n|+|V_n|\leq |x|+(1+T)|V_n|\,.
\end{equation}
Thus it suffices to prove the equi-boundedness of \eqref{loglog} for the regularised flow $V_n$. This is a byproduct of the proof of Lemma \ref{estsuper}, where we show that the constant $A$ depends on quantities which are uniformly bounded in $n$.

\subsection{Conclusion of the proof of Theorem~\ref{thm:main}: existence of Lagrangian solutions to the Vlasov-Poisson system} 

Let $f_0$ be as in Theorem \ref{thm:main}. In order to prove existence of a Lagrangian solution to system \eqref{eq:VP-dirac}-\eqref{eq:newton}, we use a compactness argument. For each $n$, we consider the initial datum $f_0^n$ defined in \eqref{eq:indata-n}, which converges to $f_0$. The result in \cite{MMP} ensures existence and uniqueness of the classical Lagrangian solution $f_n$, $(\xi_n,\eta_n)$ to the Vlasov-Poisson system with point charge 
\begin{equation}\label{eq:VP-dirac-n}\left\{
\begin{array}{l}
\partial_t f_n + v\cdot\nabla_x f_n + (E_n+F_n)\cdot\nabla_v f_n = 0\,,\\\\
f_n(0,x,v)=f_0^{n}(x,v)\,, \\\\
E_n(t,x)=\int\frac{x-y}{|x-y|^3}\rho_n(t,y)\,dy\,,\\\\
\rho_n(t,x)=\int f_n(t,x,v)\,dv\,,\\\\
F_n(t,x)=\frac{x-\xi_n(t)}{|x-\xi_n(t)|^3}\,,
\end{array}
\right.
\end{equation}
where $(\xi_n(t),\eta_n(t))$ evolves according to
\begin{equation}\label{eq:newton-n}
\left\{ \begin{array}{l}
\dot{\xi}_n(t)=\eta_n(t)\,, \\
\dot{\eta}_n(t)=E_n(t,\xi_n(t)) \,, \\
(\xi_n(0),\eta_n(0))= (\xi_0,\eta_0)\,.
\end{array}
\right.
\end{equation}
Therefore, there exists a unique flow $Z_n=(X_n, V_n):[0,T]\times \R^3\times\R^3\rightarrow \R^3\times\R^3$ associated with the vector field $(V_n, E_n(X_n)+F_n(X_n))$, such that $f^n={Z_n}_\# f_0^n$ is the push-forward of $f_0^n$ through $Z_n$, i.e. \eqref{eq:pushf}.

From Subsection \ref{subsec:flow-existence}, there exists $Z$ such that $Z_n\to Z$ in measure, with respect to $\mu=f_0\mathcal{L}^{6}$. Therefore
we define a density $f$ which is the push forward of the initial data $f_0$ through the limiting flow $Z$, i.e.
\begin{equation*}
f:=Z_{\#} f_0\,.
\end{equation*}
The aim of this subsection is to verify that the above defined $f$ is indeed a solution to \eqref{eq:VP-dirac}-\eqref{eq:newton}. In other words, we want to perform the limit $n\to\infty$ in \eqref{eq:VP-dirac-n}-\eqref{eq:newton-n} and get \eqref{eq:VP-dirac}-\eqref{eq:newton}. This will conclude the proof of Theorem~\ref{thm:main}. To this end we observe that, up to subsequences:
\begin{itemize}
\item $f_n\rightharpoonup f$ weakly in $L^1_{x,v}$ and weakly$^*$ in $L^\infty_{x,v}$, uniformly in $t$.\\
Indeed, $f_0^n\to f_0$ in $L^1_{x,v}$ and $Z_n\to Z$ in measure $\mu$. Since the latter limit is uniform in $s$ and $t$, we define the inverse of the flow $Z_n^{-1}(t,s,x,v):=Z_n(s,t,x,v)$  and observe that $Z_n^{-1}\to Z^{-1}$ in measure and therefore $\mu$-a.e., uniformly in $t$. Given $\varphi\in C_c(\R^3\times\R^3)$, we can estimate
\begin{equation*}
\begin{split}
\iint &\varphi(x,v)\,(f_n(t,x,v)-f(t,x,v))\,dx\,dv\\
&= \iint \varphi(x,v)\,\left(f_0^n(Z_n^{-1}(t,x,v))-f_0(Z^{-1}(t,x,v))\right)\,dx\,dv\\
&=\iint \varphi(Z_n(t,x,v))\,f_0^n(x,v)\,dx\,dv-\iint \varphi(Z(t,x,v))\,f_0(x,v)\,dx\,dv\\
&=\iint \left(\varphi(Z_n(t,x,v))-\varphi(Z(t,x,v))\right)\,f_0(x,v)\,dx\,dv\\
&+\iint\varphi(Z_n(t,x,v))\,(f_0^n(x,v)-f_0(x,v))\,dx\,dv\,.
\end{split}
\end{equation*}
The first term in the r.h.s. converges to zero, since $Z_n\to Z$ $\mu$-a.e. The second term also converges to zero because $\varphi$ is bounded and $f_0^n\to f_0$ in $L^1_{x,v}$.   Moreover, since $f_n$ is equi-bounded in $L^1_{x,v}\cap L^\infty_{x,v}$, uniformly in $t$,    we obtain  weak convergence  in $L^1_{x,v}$ and weak$^*$ convergence in $L^\infty_{x,v}$ of $f_n$ to $f$, uniformly in $t$.
%

\item $\rho_n\rightharpoonup\rho$ weakly in $L^1_x$. It follows from the weak $L^1_{x,v}$ convergence of $f_n$ to $f$. Moreover, thanks to Remark \ref{remark:rho}, $\rho_n \rightharpoonup\rho$  weakly in $L^{s}_x$, for some $s>3$.

\item $\partial_t f_n$ converges to $\partial_t f$ in $\mathcal{D}'$ and $v\cdot\nabla_x f_n$ converges to $v\cdot\nabla_x f$ in $\mathcal{D}'$. 

\item $E_n\to E$ uniformly. This is a consequence of Proposition \ref{prop:moments}. Indeed, the r.h.s. of equation \eqref{eq:momentsbis} is uniformly bounded in $n$. Therefore, by Proposition \ref{prop:interpolation}, $\|\rho_n\|_{L^{\frac{m+3}{3}}}$ is uniformly bounded and Proposition \ref{prop:Sobolev} yields $\{E_n\}_n$ equi-H\"{o}lder. Ascoli-Arzel\`{a} Theorem guarantees the existence of a uniformly convergent subsequence. 
The limit couple $(E,\rho)$ satisfies $E(t,x)=\int \frac{x-y}{|x-y|^3}\,\rho(t,y)\,dy$, since $E\in M^{3/2}$ and decays at infinity, while $\rho\in L^s$, for some $s>3$.

\item $E_n\cdot\nabla_v f_n\to E\cdot\nabla_v f$ in $\mathcal{D}'$. This follows by rewriting $E_n\cdot\nabla_v f_n=\dev_v (E_n\,f_n)$ and $E\cdot\nabla_v f=\dev_v (E\,f)$,  and by the facts that $E_n\to E$ uniformly  and $f_n\rightharpoonup f$ weakly in $L^1_{x,v}$.
\end{itemize}

\noindent We are left with the part of the system \eqref{eq:VP-dirac-n}-\eqref{eq:newton-n} which involves the point charge. In particular, we define 
\be
\gamma_n(t)= (\xi_n(t),\eta_n(t))
\ee
and set
\be\label{eq:xi-eta}
(\xi(t),\eta(t)):=\lim_{n\to\infty}\gamma_n(t)\,.
\ee
Observe that the limit in \eqref{eq:xi-eta} exists. Indeed, $\gamma_n(t)$ is equi-Lipschitz because of the following estimate:
\be\label{eq:gamma-n}
{\rm Lip}(\gamma_n)\leq\|\dot{\gamma}_n\|_{L^\infty}\leq \sup_t |\eta_n(t)|+\sup_t |E_n(t,\xi_n(t))|\,,
\ee
where ${\rm Lip}(\gamma_n)$ is the Lipschitz constant of $\gamma_n$. Proposition \ref{prop:charge-bounds} yields a uniform bound on the first term in the r.h.s. of \eqref{eq:gamma-n}, that combined with the uniform bounds on $E_n$ proved in this subsection, implies $\gamma_n$ equi-Lipschitz.  By Ascoli-Arzel\`{a} Theorem, there exists a subsequence $\{(\xi_{n_k}(t),\eta_{n_k}(t))\}_k$ which converges uniformly to $(\xi(t),\eta(t))$. 
To perform the limit in \eqref{eq:VP-dirac-n}-\eqref{eq:newton-n}, we observe that
\begin{itemize}
\item $(\dot{\xi}_n(t),\dot{\eta}_n(t))\to(\dot{\xi}(t),\dot{\eta}(t))$.  Indeed, $(\xi_n(t),\eta_n(t))$ converges to $(\xi(t),\eta(t))$ uniformly and
\be\label{eq:dotxi-doteta}
\sup_t|\dot{\gamma}_n(t)-(\eta(t),E(t,\xi(t)))|\leq \sup_t|\eta_n(t)-\eta(t)|+\sup_t|E_n(t,\xi_n(t))-E(t,\xi(t))|\,.
\ee
The first term in the r.h.s. of \eqref{eq:dotxi-doteta} converges to zero uniformly. As for the second term, we use that
\be\label{eq:En-E}
\begin{split}
\sup_t |E_n(t,\xi_n(t))-&E(t,\xi(t))|\\
&\leq \sup_t |E_n(t,\xi_n(t))-E(t,\xi_n(t))| + \sup_t |E(t,\xi_n(t))-E(t,\xi(t))|\\
&\leq \sup_{t,x}|E_n(t,x)-E(t,x)|+\sup_t |E(t,\xi_n(t))-E(t,\xi(t))| .
\end{split}
\ee
Combining the facts that $E_n\to E$, $\xi_n\to\xi$ and $E$ is uniformly continuous, the last line in \eqref{eq:En-E} vanishes as $n\to\infty$.

\item $F_n\to F$ in $L^1_{x,\,\rm loc}$. Indeed, $F_n\to F$ pointwise, by the uniform convergence of $\xi_n(t)$ to $\xi(t)$ up to subsequences, and $F_n,\,F\in L^1_{\rm loc}(\R^3)$. Therefore, we conclude by Dominated Convergence's Theorem.

\item $F_n\cdot\nabla_v f_n\to F\cdot\nabla_v f$ in $\mathcal{D}'$. This follows by rewriting $F_n\cdot\nabla_v f_n=\dev_v (F_n\,f_n)$ and $F\cdot\nabla_v f=\dev_v (F\,f)$,  and by the facts that $F_n\to F$ in $L^1_{\rm loc}(\R^3)$ and $f_n\stackrel{*}{\rightharpoonup}  f$ weakly$^*$ in $L^\infty_{x,v}$.
  
\end{itemize}

%
%
%

 \bigskip


\adresse


\begin{thebibliography}{10}

\bibitem{Amb} L. Ambrosio, \emph{Transport equation and Cauchy problem for BV vector fields}, Invent. Math. \textbf{158}, 227--260, 2004.

\bibitem{ACF2} L. Ambrosio, M. Colombo and A. Figalli, \emph{Existence and Uniqueness of Maximal Regular Flows for Non-smooth Vector Fields}. Arch. Rational Mech. Anal. \textbf{218}, no. 2, 1043--1081 (2015).

\bibitem{ACF} L. Ambrosio, M. Colombo and A. Figalli, \emph{On the Lagrangian structure of transport equations: the Vlasov--Poisson system}, arXiv:1412.3608

\bibitem{HW} L. Ambrosio and G. Crippa, \emph{Continuity equations and ODE flows with non-smooth velocity}, Proceedings of the Royal Society of Edinburgh: Section A Mathematics \textbf{144} (2014), n. 6, 1191--1244.

\bibitem{A} A. A. Arsenev, \emph{Global existence of a weak solution of Vlasov's system of equations}, U. S. S. R. Comput. Math. Math. Phys. \textbf{15} (1975), 131--143.

\bibitem{BD} C. Bardos and P. Degond, \emph{Global existence for the Vlasov--Poisson equation in 3 space variables with small initial data}, Ann. Inst. H. Poincar\'e Anal. Non Lin\'eaire, \textbf{2} (1985), 101--118.

\bibitem{BBC2} A. Bohun, F. Bouchut and G. Crippa, \emph{Lagrangian solutions to the Vlasov-Poisson system with $L^1$ density}, J. Differential Equations \textbf{260} (2016), no. 4, 3576--3597. 

\bibitem{BBC1} A. Bohun, F. Bouchut and G. Crippa, \emph{Lagrangian flows for vector fields with anisotropic regularity}. Annales de l'Institut Henri Poincar\'e (C) Analyse Non Lin\'eaire \textbf{33} (2016), no. 6, 1409--1429.

\bibitem{BBC3} A. Bohun, F. Bouchut and G. Crippa, \emph{Lagrangian solutions to the 2D Euler system with $L^{1}$ vorticity and infinite energy}. Nonlinear Analysis: Theory, Methods \& Applications \textbf{132} (2016), 160--172.

\bibitem{BC}  F. Bouchut and G. Crippa \emph{Lagrangian flows for vector fields with gradient given by a singular integral}, J. Hyper. Differential
Equations \textbf{10} (2013), no. 2, 235--282.


\bibitem{CM} S. Caprino and C. Marchioro, \emph{On the plasma-charge model}, Kinet. Relat. Models \textbf{3} (2) (2010), 241--254.

\bibitem{CMMP} S. Caprino, C. Marchioro, E. Miot and M. Pulvirenti, \emph{On the 2D attractive plasma-charge model}, Comm. Partial Differential Equations \textbf{37} (2012), no. 7, 1237--1272.

\bibitem{castella} F. Castella, \emph{Propagation of space moments in the Vlasov-Poisson Equation and further results}, Ann. Inst. Henri Poincar\'e \textbf{16} (1999), no. 4, 503--533.

\bibitem{CZ} Z. Chen and X. Zhang, \emph{Sub-linear estimate of large velocity in a collisionless plasma}, Commun. Math. Sciences \textbf{12} (2014), no. 2, 279--291.


\bibitem{CDL} G. Crippa and C. De Lellis, \emph{Estimates and regularity results for the DiPerna-Lions flow}, J. Reine Angew. Math. \textbf{616} (2008), 15--46.

\bibitem{CLFMNL} G. Crippa, M. C. Lopes Filho, E. Miot and H. J. Nussenzveig Lopes,
     \emph{Flows of vector fields with point singularities and the
              vortex-wave system}. Discrete Contin. Dyn. Syst., Series A
    {\bf 36},
 (2016),
    no. {5},
     {2405--2417}.
     
\bibitem{DMS} L. Desvillettes, E. Miot and C. Saffirio, \emph{Polynomial propagation of moments and global existence for a Vlasov--Poisson system with a point charge}, Ann. Inst. H. Poincar\'e (C) Anal. Non Lin\'eaire \textbf{32} (2015), no. 2, 373--400.

\bibitem{DiLi} R. J. Di Perna and P.-L. Lions, \emph{Ordinary differential equations, transport equations and Sobolev spaces}, Invent. Math. \textbf{98} (1989), 511--547.

\bibitem{E} L. C. Evans, Partial differential equations. Graduate Studies in Mathematics, vol. \textbf{19}. American Mathematical Society, Providence, Rhode Island. 


\bibitem{Gr1} L. Grafakos, 
     {Classical {F}ourier analysis},
  {Graduate Texts in Mathematics} 
 \textbf{249},
  {Third Edition},
 {Springer, New York},
      {(2014)}	

\bibitem{Gr} L. Grafakos, 
     {Modern {F}ourier analysis},
  {Graduate Texts in Mathematics} 
 \textbf{250},
  {Third Edition},
 {Springer, New York},
      {(2014)}	



\bibitem{HM} T. Holding and E. Miot, \emph{Uniqueness and stability for the Vlasov-Poisson system with spatial density in Orlicz spaces}, arXiv:1703.03046v1

\bibitem{Ior} S. V. Iordanskii, \emph{The Cauchy problem for the kinetic equation of plasma}, Trudy Mat. Inst. Steklov. \textbf{60} (1961), 181--194.

 \bibitem{LP} P.-L. Lions and B. Perthame, \emph{Propagation of moments
and regularity  for the 3-dimensional Vlasov-Poisson system}, Invent.
Math. \textbf{105} (1991), 415--430.

\bibitem{L} G. Loeper, \emph{Uniqueness of the solution to the Vlasov-Poisson system with bounded density}, J. Math. Pures Appl. (9) \textbf{86} (2006), no. 1, 68--79.

\bibitem{MMP} C. Marchioro, E. Miot and M. Pulvirenti, \emph{The Cauchy problem for the $3-D$ Vlasov-Poisson system with point charges}, Arch. Ration. Mech. Anal. \textbf{201} (2011), 1-26.

\bibitem{M} E. Miot, \emph{A uniqueness criterion for unbounded solutions to the Vlasov-Poisson system}, Commun. Math. Phys. \textbf{345} (2016), no. 2, 469--482.

\bibitem{UO} S. Okabe and T. Ukai, \emph{On classical solutions in the large in time of the two-dimensional Vlasov equation}, Osaka J. Math. \textbf{15} (1978), 245--261.


\bibitem{Pf} K. Pfaffelm{o}ser, \emph{Global existence of the Vlasov-Poisson system in three dimensions for general initial data}, J. Differ. Equ. \textbf{95} (1992), 281--303.


\bibitem{Sch} J. Schaeffer, \emph{Global existence of smooth solutions to the Vlasov-Poisson system in three dimensions}, Commun. Partial Differ. Equations \textbf{16} (1991), no. 8--9, 1313--1335.


\bibitem{W}  S. Wollman, \emph{Global in time solution to the three-dimensional Vlasov-Poisson system}, J. Math. Anal. Appl. \textbf{176}  (1996), no. 1, 76--91.



\end{thebibliography}
 \end{document}